\documentclass[final]{amsart}

\usepackage{amssymb, hyperref}
\usepackage{enumerate}
\usepackage[mathscr]{eucal}
\usepackage{xcolor}
\usepackage{MnSymbol}

\usepackage{tikz}

\usetikzlibrary{matrix,arrows}

\newtheorem{thm}{Theorem}[section]
\newtheorem{prop}[thm]{Proposition}
\newtheorem{cor}[thm]{Corollary}

\newtheorem{lemma}[thm]{Lemma}

\theoremstyle{definition}
\newtheorem{defn}[thm]{Definition}

\theoremstyle{remark}
\newtheorem{remark}[thm]{Remark}

\newtheorem{example}[thm]{Example}
\newtheorem{notation}[thm]{Notation}

\newcommand{\lvA}{{\mathbb A}}
\newcommand{\lvD}{{\mathbb D}}
\newcommand{\lvZ}{{\mathbb Z}}
\newcommand{\lvG}{{\mathbb G}}

\newcommand{\sF}{\mathscr{F}}

\newcommand{\sE}{\mathscr{E}}
\newcommand{\sP}{\mathscr{P}}
\newcommand{\sO}{\mathscr{O}}
\newcommand{\sN}{\mathscr{N}}
\newcommand{\sL}{\mathscr{L}}

\newcommand{\pardom}{{\left(\frac{1}{\vec{r}}\lvZ\right)^{op}}}

\newcommand{\B}{{\bf B}}
\newcommand{\C}{{\bf C}}

\newcommand{\Spec}{\mathrm{Spec}}
\newcommand{\Func}{{\mathrm{Func}}}
\newcommand{\Bun}{{\mathrm{Bun}}}
\newcommand{\coker}{{\mathrm{coker}}}
\newcommand{\vect}{\mathrm{Vect}}
\newcommand{\Rep}{{\mathrm{Rep}}}
\newcommand{\red}{{\rm red}}
\newcommand{\triv}{{\rm triv}}
\newcommand{\iso}{\stackrel{\sim}{\rightarrow}}
\newcommand{\rk}{{\rm rk}}
\newcommand{\op}{{\rm op}}
\newcommand{\Deck}{{\rm Deck}}
\newcommand{\para}{{\rm par}}
\newcommand{\Ob}{{\rm Ob}}

\newcommand{\ef}{\mathbf{F}}
\newcommand{\eg}{\mathbf{G}}

\begin{document}

\title{Pullback of Parabolic Bundles and Covers of ${\mathbb P}^1\setminus\{0,1,\infty \}$}
\author{Ajneet Dhillon \and Sheldon Joyner}

\address{Department of Mathematics \\ University of Western Ontario \\ London Ontario N6A 5B7}

\email{adhill3@uwo.ca}
\email{sjoyner@uwo.ca}

\subjclass[2000]{14H60}

\keywords{Parabolic vector bundles, Nori finite bundles}

\begin{abstract}We work over an algebraically closed ground field of characteristic zero. A $G$-cover of ${\mathbb P}^1$
ramified at three points allows one to assign to each finite dimensional representation   $V$ of $G$ a vector bundle $\oplus \sO(s_i)$ on 
${\mathbb P}^1$ with parabolic structure at the ramification points. This produces a tensor functor from representation of $G$ to vector
bundles with parabolic structure that characterises the original cover. This work attempts to describe this tensor functor in terms of
group theoretic data. More precisely, we construct a pullback functor on vector bundles with parabolic structure and describe the
parabolic pullback of the previously described tensor functor.
\end{abstract}

\maketitle

\section{Introduction}

We work over an algebraically closed ground field $k$ of characteristic zero. If $G$ is a finite
group then by \cite{nori:76} a $G$-torsor $f:X\rightarrow Y$ in the category of algebraic varieties 
can be thought of as a tensor functor $\Rep\text{-}G\rightarrow \vect(Y)$. Concretely the associated tensor functor sends the representation $V$
to the vector bundle $f_*(V\otimes\sO)^G$. When the cover ramifies, as was observed in \cite{nori:82}, we need to consider
tensor functors into the category of vector bundles with appropriate parabolic structure. 

In the case where $Y={\mathbb P}^1$ then we have $f_*(V\times \sO)^G=\oplus \sO(s_i)$. The integers $s_i$ are difficult to 
compute and one of our results is to find an upper bound on them when there is ramification at $0,1$ and $\infty$ only. The bound \ref{t:bound}, \ref{e:improve}, improves the known bound in \cite{borne:07}.
There is one case in which it is easy to compute the integers $s_i$, namely when the group $G$ is cyclic. Our method is a kind of reduction to the cyclic case by removing ramification at $0$. More precisely,
the endomorphism $z\mapsto z^n$ of ${\mathbb P}^1$ algebraically deloops loops around the origin. Pulling back a cover along this morphism removes ramification 
of order $n$ at the origin. To make our method work we need to define a pullback morphism for parabolic bundles. As in \cite{iyersimpson:07} and \cite{borne:07} this entails use of an equivalence of categories due to Biswas, \cite{biswas:97}, between parabolic bundles of a certain kind, and vector bundles on an associated root stack. The pullback operation is difficult to reverse, that is given a
morphism $f:X\rightarrow Y$ of smooth projective curves and a parabolic bundle $\sF_\bullet$  on $X$, to construct a parabolic bundle on $Y$ that pulls back to $\sF_\bullet$. In fact, the difficulty in reversing the parabolic pullback gives a new explanation for the fact that it is difficult to compute the $s_i$.

The interest in computing the $s_i$ lies in the following. A finite quotient $q:F_2\twoheadrightarrow G$ of the free group on two letters produces a cover $X_q\rightarrow {\mathbb P}^1$ ramified at three points. The absolute Galois group ${G_{\mathbb Q}}$ of ${\mathbb Q}$ acts faithfully on such covers. However, given $q$, the Galois action is difficult to understand, and it is not known what finite quotient of ${G_{\mathbb Q}}$ acts, sending the cover to some other non-isomorphic cover. One way to try to understand this question is to give a more algebraic construction of the cover. The theory of tannakian categories allows one to do this. One should view the cover as a tensor functor into parabolic bundles and then understand the Galois action on such tensor functors. This work should be seen as a first step towards understanding these tensor functors. In this paper we understand their parabolic pullbacks. To understand the original functor amounts to faithfully flat descent for parabolic bundles. This will be a topic of future work.

In section two we recall some results of Nori on principal bundles and tensor functors. The third section
recalls the notion of root stack introduced in \cite{cadman:06}. Section four introduces parabolic bundles in our context.
The definition here is equivalent to the one in \cite{mehta:80}. We also recall from \cite{yokogawa:95} the construction of tensor product
and internal hom for parabolic bundles.
Section five is devoted to proving the orbifold-parabolic correspondence in our
context. This result is not new and goes back to \cite{biswas:97}. The formulation here is based on the results of \cite{borne:07}.

The new results begin in section six. We describe a construction on parabolic bundles that corresponds to pullback
of orbifold bundles. In section seven we use some combinatorics to describe the case of cyclic covers. The final section
gives an upper bound on the integers $s_i$ described above, in the case of a $G$-cover of ${\mathbb P}^1\setminus\{0,1,\infty \}$.
The group $G$ need not be abelian here.

\section*{Acknowledgements}
The authors wish to thank Donu Arapura and Jochen Heinloth for very helpful advice and conversations. The parabolic pullback was originally described to
the first named author by I. Biswas through a conversation. The authors first became interested in this topic through a talk given by V. Balaji at the
University of Western Ontario in 2009.

\section*{Notations and Conventions}
\begin{enumerate}[(i)]
\item $k$ an algebraically closed field of characteristic 0.
\item $X$ a connected smooth projective curve over $k$.
\item For $x\in {\mathbb R}$ denote by $\lfloor x \rfloor$ the floor of $x$, i.e. the largest integer smaller than $x$.
\end{enumerate}

\section{Some Results of { Nori}}

In this section we recall some results from \cite{nori:76} and \cite{nori:82}. We begin by recalling the notion of a tannakian category. 
For a less terse formulation refer to \cite{saavedra:72} or \cite{deligne:82t}.

Let $L$ be a field. We denote by $\vect(L)$ the category of finite-dimensional $L$-vector spaces.

\begin{defn}A \emph{tannakian category} over $L$ consists of a quadruple $({\mathbf{C}}, \otimes, F, U)$ where 
\newline T1. ${\mathbf{C}}$ is a small, $L$-linear, abelian category.
\newline T2. $F : {\mathbf{C}} \rightarrow \vect(L)$ is an $L$-linear additive faithful exact functor called the fiber functor.
\newline T3. $\otimes : {\mathbf{C}}\times {\mathbf{C}} \rightarrow {\mathbf{C}}$ is an associative and commutative functor that is $L$-linear in each variable.
\newline T4. $U$ is a unit for $\otimes.$
\newline This data is subject to the following constraints:
\newline C1. $F$ preserves $\otimes$.
\newline C2. $F$ preserves the associativity and commutativity constraints.
\newline C3. $FU \stackrel{\sim}{\rightarrow} k.$
\newline C4. $\dim FV = 1$ if and only if there exists $V^{-1} \in \mathrm{Objects}({\mathbf{C}})$ such that
$V \otimes V^{-1}\cong U$.
\end{defn}

\begin{remark}
 One uses \cite[Proposition 1.20]{deligne:82t} to see that the category $\mathbf{C}$ is necessarily rigid.
\end{remark}

If $G$ is an affine group scheme over $k$ then the category $\Rep\text{-}G$ of finite dimensional left representations 
of $G$ is a tannakian category over $k$. In fact :

\begin{thm}
 Any tannakian category over $k$ is equivalent to $\Rep\text{-}G$ for some affine group scheme $G$ over $k$. Under this
correspondence a homomorphism of affine group schemes corresponds to a tensor functor that
commutes with fiber functor and preserves units. 
\end{thm}

For a scheme $X$ over $k$ denote by $\vect(X)$ the category of algebraic vector bundles over $X$. 
The category $\vect(X)$ is a $k$-linear tensor category. The tensor product is associative and 
commutative and has a unit. Taking the fibre over a $k$-point gives it the structure of a tannakian category.

\begin{defn}
A rigid tensor $G$-functor on $X$ is an $R$-linear exact $\otimes$-functor $F: \Rep\text{-}G \rightarrow \vect(X)$ such that
\newline F1. $F$ commutes with $\otimes$ 
\newline F2. $F$ preserves the associativity and commutativity constraint
\newline F3. $\rk FV = \dim V$
\newline F4. $F(V_{\triv}) = {\sO}_{X}$
\end{defn}

We denote the category of such functors by $\Func^{\otimes}(\Rep\text{-}G, \vect(X)).$
A morphism in this category is a natural transformation $\eta : F \rightarrow G$ such that the
following diagram commutes :

\begin{center}
\beginpgfgraphicnamed{Dessin-g1.eps}
\begin{tikzpicture}
 \node (tl) at (0,2) {$\otimes_{i\in I} F(X_i)$} ; 
 \node (bl) at (0,0) {$\otimes_{i\in I} G(X_i)$} ; 
 \node (tr) at (4,2) {$F(\otimes_{i\in I} X_i)$} ; 
 \node (br) at (4,0) {$G(\otimes_{i\in I} X_i)$.} ; 
 \draw [>=latex,->] (tl) --node[above=1pt] {$\sim$} (tr);
 \draw [>=latex,->] (bl) --node[above=1pt] {$\sim$} (br);
 \draw [>=latex,->] (tr) --node[right=1pt] {$\eta$} (br);
 \draw [>=latex,->] (tl) --node[left=1pt] {$\eta$} (bl);
\end{tikzpicture}
\endpgfgraphicnamed
\end{center}
Such a natural transformation is necessarily an isomorphism, \cite[Proposition 1.13]{deligne:82t}.

Given $P \rightarrow X,$ a $G$-torsor, we obtain a natural functor $$F_{P}\in \Func^{\otimes}(\Rep\text{-}G, \vect(X))$$ given by 
$V \mapsto P\times_{G}V.$

We denote by $\Bun_{G,X}$ the category of $G$-torsors over $X$.
Notice that all the morphisms in this category are isomorphisms.

\begin{thm}
\label{t:nori}
There is an equivalence of categories
\[
\Bun_{G,X} \stackrel{\sim}{\rightarrow}\Func^{\otimes}(\Rep\text{-}G, \vect(X)).
\]
\end{thm}
\begin{proof} 
See \cite{nori:76}. 
\end{proof}

We will mostly be interested in the case where $G$ is a finite group and $X={\mathbb P}\setminus \{0,1,\infty\} $. 
To make this setup more useful in this case we need a ramified version of this theorem. 
Such a theorem already exists in \cite{nori:82}, but we wish to restate things in terms of stacks.
For now let us record the following corollary.

\begin{cor}
\label{c:nori}
Let $H$ be another finite group acting on $X$.
 Denote by $\Bun^H_{G,X}$ the category of $G$-torsors with an action of $H$ that commutes with the action of $G$.
Then we have an equivalence of categories
$$
\Bun^H_{G,X} \stackrel{\sim}{\rightarrow}\Func^{\otimes}(\Rep\text{-}G, \vect_H(X)).
$$
Here $\vect_H(X)$ is the category of $H$-vector bundles on $X$.
\end{cor}

\begin{proof}
 Given a $G$-torsor $P\rightarrow X$ with a commuting $H$-action we obtain for each $h \in H$ a tensor 
functor 
$$
F_h : \Rep\text{-}G \rightarrow \vect(X).
$$
But as the pullbacks $P\times_{X,h} X$ are all isomorphic the functors above are all isomorphic by the
theorem so we obtain a functor into $\vect_H(X)$.

Conversely suppose that we have a tensor functor 
$$
F : \Rep\text{-}G \rightarrow \vect_H(X).
$$
Ignoring the $H$-action we obtain a torsor $P\rightarrow X$. But now the pullbacks $P \times_{X,h}X$ are all isomorphic
as the original bundles were $H$-bundles.
\end{proof}

\section{Root Stacks}
\label{RS}

In this section, we recall some constructions from \cite{cadman:06}. 

We shall implicitly make use of the following fact throughout this section : to give a morphism from 
a scheme $S$ to the quotient stack $[\lvA^k/\lvG_m^k]$ is the same as giving a tuple $(\sL_i, s_i)_{i=1}^k$ of line bundles $\sL_i$ on $S$ and sections $s_i\in\Gamma(S,\sL_i)$, see \cite[Lemma 2.1.1]{cadman:06}.

Given a $k$-tuple $\vec{r}=(r_{1}, \ldots, r_{k})$ of positive integers there is a morphism of quotient stacks
\[
\theta_{\vec{r}}:\left[ \lvA^{k}/\lvG_{m}^{k}\right]
\rightarrow
\left[ \lvA^{k}/ \lvG_{m}^{k}\right]
\]
induced by the morphism 
\[
\lvA^{k} \rightarrow \lvA^{k}
\]
\[
(x_{1}, \ldots, x_{k}) \mapsto
(x_{1}^{r_{1}}, \ldots , x_{k}^{r_{k}}).
\]

\begin{defn}
Let $\mathbb{D} = (D_{1}, \ldots, D_{k})$ be a $k$-tuple of effective Cartier divisors on a scheme $S$. This data defines a morphism
$S\rightarrow [\lvA^{k}/\lvG_{m}^{k}]$.
Define the root stack 
$S_{\mathbb{D}, \vec{r}}$ to be 
\[
S_{\mathbb{D}, \vec{r} }
= 
S\times_{\left[\lvA^{k}/\lvG_{m}^{k}\right], \theta_{\vec{r}}}\left[\lvA^{k}/\lvG^{k}_{m}\right].
\]
\end{defn}

\begin{remark} 
\label{r:lift}
Let $f : T\rightarrow S$ be a morphism. A lift of $f$ to a $T$-point of $S_{\mathbb{D}, \vec{r}}$ is the same as giving 
\[
(M_{1}, \ldots, M_{k}, t_{1}, \ldots, t_{k}, \phi_{1}, \ldots, \phi_{k})
\] 
where $M_{i}$ are line bundles on $T$, $\phi_{i}$ are isomorphisms
$M_{i}^{r_{i}} \iso 
f^{*}{\sO}(D_{i})$ and $t_{i}$ are global sections of $M_{i}$ such that 
\[ 
\phi_{i}(t_{i}^{r_{i}}) = s_{D_{i}},
\]
where $s_{D_i}$ denotes the tautological section of $\sO(D_i)$ vanishing along $D_i$.
\end{remark}

\begin{prop} \label{p:quotroot}
Let $Y$ be a smooth projective curve with an action of  a finite group $G$. Let
$\psi: Y \rightarrow Y/G = X$ be the projection and assume that the action is generically free.
 Let the ramification divisor of $\psi$ be $p_{1}+ \ldots + p_{k}$ with ramification indices $r_{1}, \ldots, r_{k}$. 
Set $\mathbb{D} = (p_{1}, \ldots , p_{k})$ and $\vec{r} = (r_{1}, \ldots, r_{k})$. Then 
\[
\left[ Y/G\right] \iso X_{\mathbb{D}, \vec{r}}.
\]
\end{prop}
\begin{proof}
 Let $\pi: X_{\mathbb{D}, \vec{r}} \rightarrow X$ be the canonical morphism.
Write
\[
\psi^{*}(p_{i}) = r_{i}D_{i}.
\]
Then the $D_{i}$ produce a $G$-equivariant morphism 
\[
\alpha: Y \rightarrow X_{\mathbb{D}, \vec{r}}.
\]
Hence the question that we have an isomorphism is local.
 
We consider an open affine $\Spec A \subset X$ with preimage $\Spec B \subset Y.$ 
We may assume $p_{1} \in \Spec A$ and $p_{i} \not \in \Spec A$ for $i >1.$ Let $s_{p_1}$ be
a parameter at $p_1$.
 Then $\pi^{-1}(\Spec A)$ is the quotient stack
\[
\left[
\Spec\left( A[t]/(t^{r_{1}} - s_{p_{1}})\right) /\mu_{r_{1}}\right],
\]
see \cite[Example 2.4.1]{cadman:06}.
We have a diagram
\begin{center}
\beginpgfgraphicnamed{Dessin-g2}
\begin{tikzpicture}
   \node (bl) at (0,0)  {$\Spec\left( A[t]/(t^{r_{1}} - s_{p_{1}})\right)$} ;
   \node (tl) at (0,2)  {$\tilde{Y}$};
   \node (br) at (4.5,0)  {$X$} ;
   \node (tr) at (4.5,2)  {$Y$} ;
    \draw [->] (tl) -- (tr);
  \draw [->] (tl) -- (bl);
   \draw [->] (bl) -- (br);
   \draw [->] (tr) -- (br);
\end{tikzpicture}
\endpgfgraphicnamed
\end{center}
where ${\tilde{Y}}$ is the normalization of $Y$ restricted to $\Spec\left( A[t]/(t^{r_{1}} - s_{p_{1}})\right)$. 
By  Abhyankar's lemma, it is a $G$-torsor and hence we obtain a morphism
\[
\Spec\left( A[t]/(t^{r_{1}} - s_{p_{1}})\right)
\rightarrow 
\left[Y / G\right].
\]
Due to the fact that the torsor $\tilde{Y}$ has a $\mu_r$-action we see that this morphism gives a morphism
\[
\beta: \left[ \Spec\left( A[t]/(t^{r_{1}} - s_{p_{1}})\right) /\mu_{r_1} \right] \rightarrow 
\left[ Y / G \right].
\]
We need to show that $\alpha \cdot \beta$ and $\beta \cdot \alpha$ are automorphisms. But this is easily checked.
\end{proof}

Consider a pair $(\mathbb{D}, \vec{r})$ with $\mathbb{D} = (n_{1}p_{1}, \ldots, n_{k}p_{k})$ and $\vec{r} = (r_{1}, \ldots, r_{k})$. We define 
\[
(\mathbb{D}, \vec{r})_{red} = 
\left((p_{1}, \ldots, p_{k}), \left(\frac{r_{1}}{d_{1}}, \ldots, \frac{r_{k}}{d_{k}}\right)\right)
\]
where $d_{i} = \gcd(n_{i}, r_{i}).$
\begin{prop}
\label{p:reduct}
There is a morphism 
\[
X_{(\mathbb{D}, \vec{r})_{red}}
\rightarrow X_{(\mathbb{D}, \vec{r})}.
\]
\end{prop}
\begin{proof}
Consider a scheme $f:S\rightarrow X$. A lift of $f$ to a point of $X_{(\mathbb{D}, \vec{r})_{red}}$
corresponds to a tuple
$$
(M_{1}, \ldots, M_{k}, t_{1}, \ldots, t_{k}, \phi_{1}, \ldots, \phi_{k}),
$$
where $M_i$ are line bundles, with global sections $t_i$ and isomorphisms
$$
\phi_i : M_i^{r_i/d_i} \iso f^*\sO_X(p_i)\qquad \phi_it_i^{r_i/d_i}=s_{p_i}.
$$
Here $s_{p_i}$ is a section vanishing at $p_i$. 

Now by \cite[Remark 2.2.2]{cadman:06}, the lifting of a morphism of stacks
$X_{(\mathbb{D}, \vec{r})_{red}} \rightarrow X$ to $X_{(\mathbb{D}, \vec{r})}$ is similar to the lifting of a morphism of schemes in that it entails the same data as given in our Remark \ref{r:lift} above.

Observe that
$$
M_i^{n_i/d_i}\qquad t_i^{n_i/d_i}\qquad \phi_i^{n_i}
$$
give the data of a morphism to 
$$
X_{(\mathbb{D}, \vec{r})}.
$$
\end{proof}

\begin{prop}
\label{p:res}
We work in the situation of proposition \ref{p:quotroot}.
Suppose that 
\[
[ Y/G ] = X_{(\mathbb{D}, \vec{r})}.
\]
Consider $f : Z \rightarrow X$ with $Z$ a smooth projective curve.
Denote by $\widetilde{f^{*}Y}$ the normalization of the fibered product
$$
Z \times_X Y.
$$
 Then 
\[
\left[ \widetilde{f^{*}Y}/G \right] = Z_{(f^{*}\mathbb{D}, \vec{r})_{red}}.
\]
\end{prop}
\begin{proof}
 By the proof of (\ref{p:quotroot}) this result will follow once  we have computed the ramification indices of the morphism
$$
\widetilde{f^*Y} \rightarrow Z.
$$
Infinitesimally locally the morphism $Y\rightarrow X$ is of the form $y\mapsto y^n$ and the morphism
$Z\rightarrow X$ is of the form $z\mapsto z^m$. The pullback is the high order cusp $y^n=z^m$. This
has $d=\gcd(n,m)$ branches in its resolution and a local calculation gives the result.
\end{proof}

We shall need the following result later :

\begin{prop}
\label{p:local}
 Every vector bundle on $X_{(\mathbb{D}, \vec{r})}$ is locally a direct sum of line bundles. Furthermore, when 
$X={\rm Spec}(R)$ with $R$ local then ${\rm Pic}(X_{p,r})$ is cyclic of order $r$ and is generated by the canonical root line bundle.
\end{prop}

\begin{proof}
 See \cite[Proposition 3.12]{borne:07} and its proof.
\end{proof}

\begin{notation}
 We will denote the canonical root line bundles on $X_{(\mathbb{D}, \vec{r})}$ by $$\sN_1, \ldots, \sN_k.$$
\end{notation}

\section{Parabolic Bundles}

Let $D = n_{1}p_{1}+ \ldots + n_{k}p_{k}$ be an effective divisor on $X$ with $p_{i} \neq p_{j}$ for $i \neq j$ and $n_i\ge 0$. 
We denote by $\lvD$ the tuple $(n_1p_1, n_2p_2, \ldots, n_kp_k)$.
Fix a tuple of integers $\vec{r} = (r_{1}, \ldots, r_{k})$ with $r_{i} \geq 1.$ The set 
\[
\frac{1}{r_{1}}\lvZ \times \ldots \times \frac{1}{r_{k}}\lvZ
\]
has a natural partial ordering with
\[
\left(\frac{x_{1}}{r_{1}}, \ldots, \frac{x_{k}}{r_{k}}\right) \leq
 \left(\frac{y_{1}}{r_{1}}, \ldots, \frac{y_{k}}{r_{k}}\right) 
 \]
 if and only if
 \[
 \frac{x_{i}}{r_{i}} \leq \frac{y_{i}}{r_{i}}
 \]
 for all $i$. We shall often denote the poset 
\[
\frac{1}{r_{1}}\lvZ \times \ldots \times \frac{1}{r_{k}}\lvZ
\]
by
\[
 \frac{1}{\vec{r}} \lvZ.
\]
 If $\vec{\alpha} = (\alpha_{1}, \ldots, \alpha_{k}) \in \frac{1}{\vec{r}}\mathbb{Z}$ then there is a natural shift functor $[\vec{\alpha}]$ on the category of functors
 \[
 \left(\frac{1}{r_{1}}\lvZ \times \ldots \times \frac{1}{r_{k}}\lvZ\right)^{op} \rightarrow \vect(X)
 \]
given by precomposition with the addition functor
\[
 +\vec{\alpha} : \frac{1}{\vec{r}} \lvZ\rightarrow \frac{1}{\vec{r}} \lvZ.
\]
 
 \begin{defn}
 \label{d:pb}
 A parabolic bundle supported on $\mathbb{D}$ with $\vec{r}$-divisible weights is a functor
 \[
 {\sF}_{\bullet}:\left(\frac{1}{r_{1}}\lvZ \times \ldots \times \frac{1}{r_{k}}\lvZ\right)^{op} \rightarrow \vect(X)
 \]
 with natural isomorphisms 
 \[
 j_{\sF_{\bullet},i}: {\sF}_{\bullet} \otimes {\sO}(-n_{i}p_{i})
\iso
 {\sF}_{\bullet}[0, \ldots, 0, 1, 0, \ldots, 0]
 \]
 (with 1 in the $i$th position) making the following diagram commute
 \begin{center}
\beginpgfgraphicnamed{Dessin-g3}
\begin{tikzpicture}
   \node (b) at (2,0)  {${\sF}_{\bullet}$} ;
   \node (tl) at (0,2)  {${\sF}_{\bullet}(-n_{i}p_{i})$};
   \node (tr) at (4,2)  {${\sF}_{\bullet}[0, \ldots, 0, 1, 0, \ldots, 0]$} ;
    \draw [->] (tl) -- (tr);
  \draw [->] (tl) -- (b);
  \draw [->] (tr) -- (b);
\end{tikzpicture}
\endpgfgraphicnamed
\end{center}
This data is required to satisfy the following axioms:
\begin{enumerate}[(i)]
 \item If ${\alpha}_{i} \leq {\alpha}_{i}'\leq {\alpha}_{i}+1$  for all $i$ then 
 $\coker ({\sF}_{\vec{\alpha}'} \hookrightarrow {\sF}_{\vec{\alpha}})$ is a locally free ${\sO}_{D}$-module. 
Here $\vec{\alpha}= (\alpha_{1}, \ldots, \alpha_{k})$ and 
 $\vec{\alpha}' = (\alpha_{1}', \ldots, \alpha_{k}')$.
 \item For every $\vec{\alpha}= (\alpha_{1}, \ldots, \alpha_{k}) \in \frac{1}{\vec{r}}\lvZ$ we have that
 $\sF_{\vec{\alpha}}$ is the fibered product
of $\sF_{(\lfloor \alpha_1 \rfloor,\ldots,\lfloor \alpha_{i-1} \rfloor,\alpha_i,\lfloor \alpha_{i+1} \rfloor,\ldots,\lfloor \alpha_{k} \rfloor)}$ 
over $\sF_{(\lfloor \alpha_1 \rfloor, \ldots, \lfloor \alpha_k \rfloor)}$, i.e
$$
\sF_{\vec{\alpha}} = \bigtimes_{\sF_{(\lfloor \alpha_1 \rfloor, \ldots, \lfloor \alpha_k \rfloor)}} \sF_{(\lfloor \alpha_1 \rfloor,\ldots,\lfloor \alpha_{i-1} \rfloor,\alpha_i,\lfloor \alpha_{i+1} \rfloor,\ldots,\lfloor \alpha_{k} \rfloor)}
$$
\end{enumerate}
\end{defn}
When the context is clear, we write $j_{\sF_{\bullet},i} = j_{i}.$ The morphisms making up the functor
$$
\vec{\alpha} \le \vec{\beta} \qquad \sF_{\vec{\beta}} \rightarrow \sF_{\vec{\alpha}}
$$
are necessarily injective so the second axiom merely asserts that
$$
\sF_{\vec{\alpha}} = \bigcap \sF_{(0,\ldots,0,\alpha_i,0,\ldots,0)},
$$
when $\alpha_i>0$ and the intersection is as submodules of 
$$
\sF_{(0,0,\ldots,0)}.
$$

\begin{remark}
\label{r:ms}
When the underlying divisor is reduced, this definition is equivalent to the original definition of {Mehta} and {Seshadri} in \cite{mehta:80}. 
To spell things out, a {Mehta-Seshadri} parabolic bundle with $\vec{r}$-divisible weights and parabolic structure along $\lvD$ 
consists of a vector bundle $\sE$ and for each $p_{i}$ a filtration of 
\[
\sE_{n_ip_{i}}:=\sE_{p_{i}}\otimes \sO_{X,p_{i}}/\mathfrak{m}_{p_{i}}^{n_i}
\]
given by
\[
\sE_{n_ip_{i}}=F_{1,i}(\sE_{n_ip_{i}}) \supsetneq \ldots \supsetneq
F_{m_{p_{i} },i}(\sE_{n_ip_{i}}) \supsetneq F_{m_{p_{i}}+1,i}(\sE_{n_ip_{i}})
=0
\]
and rational numbers $(\alpha_{i,j})_{1 \leq j \leq m_{p_{i}}}$ of the form $l/r_{i}$ satisfying 
\[
0 \leq \alpha_{i,1}< \ldots <\alpha_{i, m_{p_{i}}}<1
\]
subject to the condition that 
\[
F_{j,i}(\sE_{n_ip_{i}})/F_{j+1,i}(\sE_{n_ip_{i}})
\]
be locally free as modules over $\sO_{X,p_i}/\mathfrak{m}_{p_i}^{n_i}$.
 
Let $\sF_{\bullet}$ be a parabolic bundle as defined in \ref{d:pb}. The quotients
\[
\sF_{(0, \ldots, 0,l/r_{i},0 \ldots, 0)}
/
\sF_{(0, \ldots, 0, 1, 0, \ldots, 0)}
\]
for $0\leq l/r_{i} <1$ define a filtration 
\[
F_{1,i}(\sF_{\bullet}) \supsetneq F_{2,i}(\sF_{\bullet})
\supsetneq \ldots \supsetneq F_{n_{i},i}(\sF_{\bullet}) \supsetneq 0
\]
of $\sF_{(0, \ldots,   0)}/\sF_{(0, \ldots, 0,1,0 \ldots, 0)} = \sF_{(0, \ldots , 0)} \otimes \sO(-n_ip_i)$.
We attach weights 
$\alpha_{i,j}$ to  $F_{j,i}(\sF_{\bullet})$ by setting $\alpha_{i,j} = l/r_{i}$ where $l$ is maximal such that
\[
F_{j,i}(\sF_{\bullet})=\sF_{(0, \ldots, 0, l/r_{i}, 0, \ldots, 0)}/
\sF_{(0, \ldots, 0, 1, 0, \ldots, 0)}.
\]
The process is clearly reversible. 
\end{remark}

\begin{defn}
 A morphism of parabolic bundles is a natural transformation
 \[
 \phi: {\sF}_{\bullet} \rightarrow {\sF}'_{\bullet}
 \]
 such that the following diagram commutes:
 
\begin{center}
\beginpgfgraphicnamed{Dessin-g4}
\begin{tikzpicture}
   \node (bl) at (0,0)  {${\sF'}_{\bullet}(-n_ip_i)$} ;
   \node (tl) at (0,2)  {${\sF}_{\bullet}(-n_ip_i)$};
   \node (br) at (4.5,0)  {${\sF'}_{\bullet}[0,\ldots, 0, 1, 0, \ldots, 0]$} ;
   \node (tr) at (4.5,2)  {${\sF}_{\bullet}[0,\ldots, 0, 1, 0, \ldots, 0]$} ;
    \draw [->] (tl) --node[above=1pt] {$\sim$} (tr);
  \draw [->] (tl) -- (bl);
   \draw [->] (bl) --node[above=1pt] {$\sim$} (br);
   \draw [->] (tr) -- (br);
\end{tikzpicture}
\endpgfgraphicnamed
\end{center}
\end{defn}

Denote by $\vect_{\para}(\mathbb{D}, \vec{r})$ the category of $\vec{r}$-divisible parabolic bundles with parabolic structure along 
$\mathbb{D}.$ By modifying constructions and arguments given in \cite{yokogawa:95}, it is possible to endow this category with the 
structure of rigid tensor category. This entails defining a suitable tensor product and internal hom, which we describe now.
 
We have an addition bifunctor
\[
 + : \left(\frac{1}{\vec{r}}  \lvZ\right) ^\op \times \left(\frac{1}{\vec{r}}\lvZ\right)^\op \rightarrow  \left(\frac{1}{\vec{r}}\lvZ\right)^\op
\]

\begin{defn}
  Let $\sE_\bullet$, $\sF_\bullet$ and $\sP_\bullet$ be parabolic bundles. There is hence a functor
\[
 \sE_\bullet \oplus \sF_\bullet :\left(\frac{1}{\vec{r}} \lvZ\right)^\op \times \left(\frac{1}{\vec{r}} \lvZ\right)^\op \rightarrow \vect(X).
\]
A \emph{bilinear} morphism from
$\sE_\bullet$ and $\sF_\bullet$ to $\sP_\bullet$ is a natural transformation 
$$
\eta: \sE_\bullet \oplus \sF_\bullet  \rightarrow \sP_\bullet \circ +
$$
such that for every local section $f\in F_{\vec{\alpha}}$ (resp. $e\in E_{\vec{\alpha}}$) there is a parabolic
morphism induced from $\eta$
$$
 \sE_\bullet \rightarrow \sP[\vec{\alpha}]_{\bullet}\qquad (\text{resp. } 
 \sF_\bullet \rightarrow \sP[\vec{\alpha}]_{\bullet}).
$$
\end{defn}

As above, let $\vec{\alpha}$ denote $(\alpha_{1}, \ldots, \alpha_{k})$ and similarly for $\vec{\beta}$ and $\vec{\gamma}$.
\begin{defn}
\label{def0}
Given parabolic bundles $\sE_{\bullet}$ and $\sF_\bullet$ in $\Ob(\vect_{\para}(\mathbb{D}, \vec{r}))$,
define a functor 
\[
(\sE_\bullet\otimes \sF_\bullet)_\bullet
:\left(\frac{1}{\vec{r}}\lvZ \right)^{op}
\to \vect{X}
\]
by setting
\[
(\sE_\bullet \otimes \sF_\bullet)_{\vec{\alpha}}:=
\left(\bigoplus_{\beta+\gamma = \alpha}
\sE_{\vec{\beta}} \otimes_{\sO_{X}}\sF_{\vec{\gamma}}\right)/R_{\vec\alpha}
\]
where $R_{\vec\alpha}$ is the $\sO_{X}$ submodule of 
the direct sum, which is locally generated by the sections:
\[
[\sE_{\bullet}(\vec\beta \to \vec\beta')]x \otimes y
-
x \otimes [\sF_{\bullet}(\vec\gamma' \to \vec\gamma)]y
\]
for any $\vec\beta+\vec\gamma = \vec\beta'+\vec\gamma' = \vec\alpha$
where $x \in \sE_{\vec{\beta}}$, $y \in \sF_{\vec{\gamma}'}$ and $[\sE_{\bullet}(\vec\beta \to \vec\beta')]$ denotes the morphism in $\vect(X)$ which is the image of the morphism $\vec\beta \to \vec\beta'$ in $\left(\frac{1}{\vec{r}}\lvZ \right)^{op}$ under the functor $\sE_{\bullet}$ (similarly for $[\sF_\bullet (\vec\gamma' \to \vec\gamma)])$;  and 
\[
x- j_{i}^{\vec\beta, \vec\gamma}x
\]
for $i=1,\ldots, k,$ where $j_{i}^{\vec\beta, \vec\gamma}$ denotes the morphism 
\[
(1 \otimes j_{\sF_{\bullet}, i}(\vec\gamma))\circ
(j_{\sE_{\bullet},i}(\vec\beta - (0,\ldots, 0,1,0,\ldots, 0))^{-1}\otimes 1)
\]
mapping
\begin{eqnarray*}
\sE_{\vec{\beta}} \otimes \sF_{\vec{\gamma}}
&\to &
\sE_{(\beta_{1}, \ldots, \beta_{i-1}, \beta_{i}-1, \beta_{i+1}, \ldots, \beta_{k})}\otimes \sO(-n_{i}p_{i}) \otimes
\sF_{\vec{\gamma}}
\\
&\to&
\sE_{(\beta_{1}, \ldots, \beta_{i-1}, \beta_{i}-1, \beta_{i+1}, \ldots, \beta_{k})}\otimes
\sF_{(\gamma_{1}, \ldots, \gamma_{i-1}, \gamma_{i}+1, \gamma_{i+1}, \ldots, \gamma_{k})}.
\end{eqnarray*}
Also define
the morphism 
$
\psi_{(\sE\otimes\sF)_{\bullet}}^{\vec{\alpha}, {\vec{\alpha}}'}:=(\sE\otimes \sF)_{\bullet}({\vec{\alpha}} \to\vec{\alpha}')
$
from $(\sE\otimes \sF)_{\vec\alpha}$ to 
$(\sE\otimes \sF)_{\vec\alpha'}$ in $\vect(X)$
by specifying for local sections $x \in \sE_{\vec{\beta}}$ and $y \in \sF_{\vec{\gamma}}$ with $\vec{\beta}+\vec{\gamma} = \vec\alpha$, that
\begin{eqnarray*}
\psi_{(\sE\otimes \sF)_{\bullet}}^{\vec\alpha, \vec\alpha'}(x \otimes y \mod R_{\vec\alpha}) 
&=&
([\sE_{\bullet}(\vec\beta \to \vec\alpha'-\vec\gamma)]x)\otimes y \mod R_{\vec\alpha'}
\\
&=&
x \otimes ([\sF_{\bullet}(\vec\gamma \to \vec\alpha' - \vec\beta)]y) \mod R_{\vec\alpha'}.
\end{eqnarray*}
\end{defn}

Now for each $i$, it is possible to define the isomorphism $j_{i}$ associated to the functor 
$(\sE\otimes \sF)_{\bullet}$ as follows:
Consider for $i =1, \ldots, k$, 
\[
J_{\vec\alpha}^{i}:=\bigoplus_{\vec\gamma}(1 \otimes j_{\sF_{\bullet},i}(\vec\gamma))
\]
mapping
\[
\bigoplus_{\vec\gamma} \sE_{(\vec\alpha - \vec\gamma)} \otimes 
\sF_{\vec\gamma} \otimes \sO(-n_{i}p_{i})
\to
\bigoplus _{\vec\gamma}\sE_{(\vec\alpha - \vec\gamma)}\otimes {\sF}_{(\gamma_{1}, \ldots, \gamma_{i-1}, \gamma_{i}+1, \gamma_{i+1}, \ldots, \gamma_{k})}.
\]
Then $J_{\vec\alpha}^{i}(R_{\vec\alpha}\otimes \sO(-n_{i}p_{i})
) = R_{(\alpha_{1}, \ldots, \alpha_{i}+1, \ldots, \alpha_{k})}.$ Hence  $J_{\bullet}^{i}$ descends to the quotient and we denote this morphsim $j_{(\sE_\bullet\otimes \sF_\bullet)_{\bullet}, i}$.

\begin{lemma}
With this data, $(\sE_\bullet \otimes \sF_\bullet)_{\bullet}$ is a parabolic bundle 
with a bilinear morphism 
$$
\sE_{\bullet} \oplus \sF_{\bullet} \rightarrow (\sE_\bullet \otimes \sF_\bullet)_{\bullet} \circ +
$$
that is universal for all bilinear morphisms.
\end{lemma}
\begin{proof}
It is easy to check that 
$((\sE_\bullet \otimes \sF_\bullet)_{\bullet}, j_{(\sE_\bullet \otimes \sF_\bullet)_{\bullet}, i}) \in \Ob(\vect_{\para}(\mathbb{D}, \vec{r})).$

To see the universal property, notice as in \cite{yokogawa:95} that the canonical maps
\[
f_{\vec\alpha, \vec\beta}: \sE_{\vec\alpha}
\otimes _{\sO_{X}} 
\sF_{\vec\beta}
\to (\sE_\bullet\otimes \sF_\bullet)_{\vec\alpha+\vec\beta}
\]
determine a canonical bilinear morphism
\[
f_{\bullet, \bullet}:
\sE_{\bullet} \oplus\sF_{\bullet}
\to
(\sE_\bullet \otimes \sF_\bullet)_{\bullet} \circ +
\]
of $\sE_\bullet$ and $\sF_\bullet$ to $(\sE_\bullet \otimes \sF_\bullet)_{\bullet}$
via morphisms $f_{\bullet, \vec\beta}:\sE_{\bullet} \to (\sE_\bullet \otimes \sF_\bullet)[\vec\beta]_{\bullet}$ and $f_{\vec\alpha, \bullet}: \sF_{\bullet} \to (\sE_\bullet \otimes \sF_\bullet)[\vec\alpha]_{\bullet}$ defined respectively for each fixed local section $b \in \sF_{\vec\beta}$ and $a \in \sE_{\vec\alpha}$. Because the latter morphisms are canonical embeddings, it follows that any bilinear morphism of $\sE_{\bullet}$ and $\sF_{\bullet}$ to some parabolic bundle $\mathscr{P}_{\bullet}$ factors uniquely through $(\sE_\bullet \otimes \sF_\bullet)_{\bullet} \circ +$.
\end{proof}

\begin{defn}
\label{def1}
Given parabolic bundles $\sE_\bullet$ and $\sF_\bullet$ in $\Ob(\vect_{\para}(\mathbb{D}, \vec{r}))$,
define a functor
\[
\mathscr{H}om(\sE_\bullet, \sF_\bullet)_{\bullet}:
\left(\frac{1}{\vec{r}}\mathbb{Z}\right)^{op}
\to
\vect(X)
\]
by setting
\[
\mathscr{H}om(\sE_\bullet, \sF_\bullet)_{\vec\alpha}:=
\mathscr{H}om(\sE_{\bullet}, \sF[\vec\alpha]_{\bullet}),
\]
the (vector bundle of) natural transformations from the functor $\sE_{\bullet}$ to the shifted functor $\sF[\vec\alpha]_{\bullet}.$
The morphism $\vec\alpha \to \vec\beta$ in $\left(\frac{1}{\vec{r}}\mathbb{Z}\right)^{op}$ induces a natural transformation of $\sF[\vec\alpha]_{\bullet}$ to $\sF[\vec\beta]_{\bullet}$
(i.e. the shift $[\vec\beta-\vec\alpha]$), thereby inducing a natural transformation 
\[
\mathscr{H}om(\sE_\bullet, \sF_\bullet)_{\vec\alpha} \rightarrow 
\mathscr{H}om(\sE_\bullet, \sF_\bullet)_{\vec\beta}
\]
which we regard as the image of $\vec\alpha \to \vec\beta$ under the functor $\mathscr{H}om(\sE_\bullet, \sF_\bullet)_{\bullet}.$
\end{defn}

\begin{lemma}
For a given $\mathbb{D}$ and $\vec{r},$
$\vect_{\para}(\mathbb{D}, \vec{r})$ with the tensor product and internal hom defined above in \ref{def0} and \ref{def1} respectively, is a rigid tensor category.
\end{lemma}
\begin{proof}
This follows from the same arguments used to prove Lemmas 3.5  
and 3.6 (equation (3.2))
in \cite{yokogawa:95}, modified to accord with our definitions.
\end{proof}

An alternative description of the tensor product was given in \cite{balaji:02}. This is useful for computations, so for later use, we formulate it here. The definition hinges on the embedding $\tau: X \setminus D \rightarrow X$:
\begin{defn}
\label{d:BBN}
The {BBN} tensor of the parabolic bundles $\sE_{\bullet}$ and $\sF_{\bullet}$ is the functor
\[
(\sE_\bullet \otimes \sF_\bullet)^{BBN}_{\bullet} : \pardom \rightarrow \vect(X)
\]
sending $\vec\alpha$ to the subsheaf of $\tau_{*}\tau^{*}(\sE_{\bullet}\otimes \sF_{\bullet})$ generated by (the canonical images of) $\sE_{\vec\beta} \otimes \sF_{\vec\gamma}$ for all $\vec\beta+\vec\gamma = \vec\alpha.$
\end{defn}
Since $\sE_\bullet$ and $\sF_\bullet$ are parabolic, the requisite axioms are automatically satisfied. To show that the {BBN} tensor gives a parabolic bundle, it remains to exhibit the isomorphisms $j_{i}$. Instead, we prove
\begin{lemma}
\label{l:tensor}
For any $\vec\alpha \in \pardom,$ and any parabolic bundles $\sE_{\bullet}$ and $\sF_{\bullet},$
\[
(\sE_\bullet\otimes \sF_\bullet)_{\vec\alpha} \simeq
(\sE_\bullet \otimes \sF_\bullet)^{BBN}_{\vec\alpha}.
\] 
\end{lemma}
\begin{proof}
Any bundle $\sE_{\vec\beta} \otimes \sF_{\vec\gamma}$ with $\vec\beta + \vec \gamma = \vec\alpha$ maps into $\tau_{*}\tau^{*}(\sE_\bullet \otimes \sF_\bullet)$, producing a mapping 
\[
\phi:
\oplus_{\vec\beta+\vec\gamma = \vec\alpha}\sE_{\vec\beta} \otimes \sF_{\vec\gamma}
\to
(\sE_\bullet \otimes \sF_\bullet)_{\alpha}^{BBN}
\]
which by construction is a surjection. It remains to show that $R_{\vec\alpha} = \ker \phi.$ Since these are sheaves, the question is local. It is then immediate from the definition of $R_{\vec\alpha}$ in terms of local sections, that this sheaf is a subsheaf of the kernel. An induction argument shows the reverse inclusion: Let $m$ denote the number of non-zero entries in a given element of the direct sum.  Also, let $(x_{\beta\gamma})_{\beta\gamma}$ denote an element of the direct sum, where $x_{\beta\gamma}$ is a local section of $\sE_{\vec\beta}\otimes \sF_{\vec\gamma}.$ 
Elements of the kernel for which $m=2$ are in $R_{\vec\alpha}$: If $(x_{\beta\gamma})_{\beta\gamma}$ is such an element, then denote the non-zero entries by
$x_{st}$ and $x_{uv}.$ Here suppose firstly that $x_{st} = x_{s}\otimes x_{t}$ and $x_{uv} = x_{u}\otimes x_v$ - i.e. each is a pure tensor of local sections. Then the image under $\phi$ is $\phi(x_{st})+\phi(x_{uv}) = 0.$ Abusing notation, this means that $x_s \otimes x_t = -x_u \otimes x_v$, which necessarily admits an expression as $\sE[u \to s](-x_u) \otimes x_t = (-x_u)\otimes \sF[t \to v]x_t$ so that $(x_{\beta\gamma}) \in R_{\vec\alpha}.$ More generally, if the non-zero terms are not pure tensors, by choosing bases for the local sections, which give canonical bases for the tensor products, it is possible to carry out a similar argument.  Now if it is known that elements of the kernel for which $m \leq n-1$ all lie in $R_{\vec\alpha},$   
the same is true for those with $m=n$.
To show this, we remark that because of axiom (ii) of Definition \ref{d:pb}, it suffices to consider $\vec\alpha$ of the form of $(0, \ldots, 0, a, 0, \ldots, 0)$ for some $a.$ Without loss of generality, we may suppose that $k=2$ - i.e. the tuples $(a,0)$ and $(0,b)$ need only be considered. Then for pure tensors as before, we obtain
$x_{s_1}\otimes x_{t_{1}} + \ldots +x_{s_{n-1}}\otimes x_{t_{n-1}} = -x_{s_{n}}\otimes x_{t_{n}}$ with $x_{s_{j}}$ (resp. $x_{t_{j}}$) a local section of $\sE_{s_{j}}$ (resp. $\sF_{t_{j}}$). But by adding suitable elements of $R_{\vec\alpha}$ to each term, when $\vec\alpha = (a,0),$ we may assume that the $s_{j}=(s_{j}',0)$ and the $t_{j}=(t_{j}',0).$ We may take $s_{1}'< \ldots <s_{n}'$, so that $t_{n}'< \ldots < t_{1}'.$ But then $\sE_{s_{1}}\supset \ldots \supset \sE_{s_{n}}$ while $\sF_{t_{1}} \subset \ldots \subset \sF_{t_{n}}.$ Consequently $x_{s_{n}} \otimes x_{t_{n}} \in \sE_{s_{1}} \otimes \sF_{t_{n-1}},$ so that $x_{s_{n}} \otimes x_{t_{n}} - \sE[s_{n}\to s_{n-1}]x_{s_{n}} \otimes x_{t_{n}} \in R_{\vec\alpha}$, and may be added to the right side to reduce to the case that $m=n-1$. The general case may be handled using local bases as before.

\end{proof}

We define a parabolic bundle $\sO_{X\bullet}:\pardom\rightarrow \vect(X)$ by setting
\[
\begin{array}[c]{cccc}
 \sO_{X\;(0,\ldots,0)} & = & \sO_X &\\
 \sO_{X\;(0,\ldots,0,t,0,\ldots,0)} & = & \sO_X(-np_i)& \mbox{for}\;t \in (0,1].
\end{array}
\]
It is easily seen that this bundle is a unit for the tensor product.

\section{The Parabolic -  Orbifold Correspondence}

Recall that ${\mathscr{N}}_{1}, \ldots, {\mathscr{N}}_{k}$ denote  the canonical line bundles on $X_{\mathbb{D}, \vec{r}}$ that 
are roots of ${\sO}(n_{i}p_{i}).$ Following \cite{biswas:97} and \cite{borne:07} we then define a functor
\[
\ef_{\mathbb{D}, \vec{r}}: \vect(X_{\mathbb{D}, \vec{r}}) \rightarrow \vect_{\para}(\mathbb{D}, \vec{r})
\]
\[
{\sF}
\mapsto 
\left[ 
\left(
\frac{l_{1}}{r_{1}}, \ldots, \frac{l_{k}}{r_{k}}
\right) \mapsto
\pi_{*}(
{\mathscr{N}}_{1}^{-l_{1}} \otimes \cdots \otimes {\mathscr{N}}_{k}^{-l_{k}}\otimes {\sF}) 
\right].
\]

\begin{remark}
This functor is in fact a tensor functor where the tensor product in the category of parabolic bundles is defined as in the last section.
In order to prove this it is useful to use the description of the tensor product in \cite{balaji:02}. Given two vector bundles 
$\sF_1$ and $\sF_2$ we need to
show that the two parabolic bundles $\ef(\sF_1\otimes \sF_2)$ and $\ef(\sF_1)\otimes\ef(\sF_2)$ are isomorphic. Away from the support of $\mathbb{D}$ the stack $X_{\mathbb{D}, \vec{r}}$ is isomorphic to the curve $X$. Hence both of these bundles are subbundles of $\tau_*\tau^*(\ef(\sF_1)\otimes \ef(\sF_2))$. We need to
show that they are the same subbundle. This question is local so we reduce to the case of one parabolic point and $\sF_i=\sN^{a_i}$. This is now
easily checked.
\end{remark}

The main result of this section is:
\begin{thm}
\label{t:correspondence}
The functor $\ef_{\mathbb{D}, \vec{r}}$ is an equivalence of categories.
\end{thm}

The proof given below is entirely analogous with the proof given in \cite{borne:07}.

We have a canonical isomorphism
\[
\pi^{*}{\sO}^{\alpha}(n_{i}p_{i}) \rightarrow {\mathscr{N}}_{i}^{\alpha r_{i}}
\]
and a section
\[
s \in \Gamma(X_{\mathbb{D}, \vec{r}}, {\mathscr{N}}_{i}).
\]
This produces by adjointness a canonical morphism
\[
{\sO}(n_{i}p_{i})^{\lfloor l/r_{i}\rfloor} \rightarrow
\pi_{*}({\mathscr{N}}_{i}^{l}).
\]
\begin{prop}
\label{p:push}The above morphism is an isomorphism.
\end{prop}

\begin{proof} See \cite[3.11]{borne:07}. \end{proof}

To proceed we need to recall the notion of a \emph{universal wedge} in category
theory. Let $\B$ and $\C$ be categories and consider a functor $F:\B^\op\times\B\rightarrow \C$.
A \emph{wedge} of $F$ is an object $x$ of $\C$ and
a collection of morphisms $a_i:F(i,i)\rightarrow x$ which are \emph{dinatural}, that is
for every morphism $f:i\rightarrow j$ in $\B$ the following diagram commutes
\begin{center}
\beginpgfgraphicnamed{Dessin-g5}
\begin{tikzpicture}
   \node (l) at (0,2)  {$F(j,i)$} ;
   \node (t) at (4,4)  {$F(i,i)$};
   \node (b) at (4,0)  {$F(j,j)$} ;
   \node (r) at (8,2)  {$x$};
   \draw [->] (l) --node[above=5pt] {$F(f^{op},1)$} (t);
   \draw [->] (l) --node[above=2.5pt] {$F(1,f)$} (b);
   \draw [->] (t) --node[above=1pt] {$a_i$} (r);
   \draw [->] (b) --node[above=1pt] {$a_j$} (r);
\end{tikzpicture}
\endpgfgraphicnamed
\end{center}

A smallest such wedge is called a universal wedge. If it exists we will denote it by
$\int^IF(I,I)$.

\begin{prop}
\label{p:uniwedge}
Let ${\sF}_{\bullet} \in \vect_{\para}(\mathbb{D}, \vec{r}).$ The universal wedge
\[
\int^{
\frac{1}{\vec{r}}\lvZ
} 
{\mathscr{N}}_{1}^{l_{1}} \otimes \cdots \otimes {\mathscr{N}}_{k}^{l_{k}}
\otimes 
\pi^{*}{\sF}_{\left(\frac{l_{1}}{r_{1}}, \cdots ,
\frac{l_{k}}{r_{k}}\right)}
\]
exists in $\vect(X_{(\mathbb{D}, \vec{r})})$.
\end{prop}
\begin{proof}
The question is local as wedges are colimits. The proof in the local case is already in \cite{borne:07}. 
\end{proof}

We denote the functor arising from \ref{p:uniwedge} by $\eg_{\mathbb{D}, \vec{r}}.$

\begin{prop} Let $\sF \in \vect(X_{\lvD, \vec{r}}).$
The natural map
\[
{\mathscr{N}}^{l_{1}}_{1} \otimes \cdots \otimes {\mathscr{N}}_{k}^{l_{k}} \otimes
\pi^{*}\pi_{*}({\mathscr{N}}^{-l_{1}}_{1} \otimes \cdots \otimes {\mathscr{N}}_{k}^{-l_{k}} \otimes {\sF}) \rightarrow {\sF}
\]
is dinatural in $(l_{1}, \ldots, l_{k}).$
\end{prop}
\begin{proof}
 The morphism in question comes by tensoring the counit of adjunction
\[
 \pi^*\pi_*(\sN_1^{-l_1}\otimes \ldots \otimes \sN_k^{-l_k} \otimes \sF) \rightarrow 
  \sN_1^{-l_1}\otimes \ldots \otimes \sN_k^{-l_k} \otimes \sF.
\]
It is relatively straightforward to show that the resulting morphism is dinatural. The details are
spelled out in \cite[Lemma 3.18]{borne:07}.
\end{proof}

\begin{cor}
\[
\eg_{\mathbb{D}, \vec{r}}\circ \ef_{\mathbb{D}, \vec{r}}\simeq 1.
\]
\end{cor}
\begin{proof}
By the proposition, there exists a natural transformation
\[
\eg_{\lvD, \vec{r}} \circ \ef_{\lvD, \vec{r}}\rightarrow 1.
\]
To show that it is an isomorphism we may argue locally. This argument can be 
found in \cite[page 18]{borne:07}.
\end{proof}

Finally we need to show that 
\[
\ef_{\mathbb{D}, \vec{r}} \circ \eg_{\mathbb{D}, \vec{r}}\simeq 1.
\]

We have 
\begin{eqnarray*}
&&
\pi_{*}\left({\mathscr{N}}_{1}^{-m_{1}}\otimes \cdots \otimes {\mathscr{N}}_{k}^{-m_{k}}\otimes \int {\mathscr{N}}_{1}^{l_{1}}\otimes \cdots \otimes {\mathscr{N}}_{k}^{l_{k}} \otimes \pi^{*}{\sF}_{\left(\frac{l_{1}}{r_{1}}, \ldots, \frac{l_{k}}{r_{k}}\right)}\right)
\\
&\simeq&
\pi_{*}\left(\int {\mathscr{N}}_{1}^{l_{1}-m_{1}}\otimes \cdots \otimes {\mathscr{N}}_{k}^{l_{k}-m_{k}}\otimes \pi^{*}{\sF}_{\left(\frac{l_{1}}{r_{1}}, \ldots, \frac{l_{k}}{r_{k}}\right)}\right)\\
&\simeq &
\int \pi_{*}\left({\mathscr{N}}_{1}^{l_{1}-m_{1}}\otimes \cdots \otimes {\mathscr{N}}_{k}^{l_{k}-m_{k}}\otimes \pi^{*}{\sF}_{\left(\frac{l_{1}}{r_{1}}, \ldots ,\frac{l_{k}}{r_{k}}\right)}\right) \; \; \mbox{$\pi_{*}$ is exact}
\\
&\simeq & \int \pi_{*}({\mathscr{N}}_{1}^{l_{1}-m_{1}}\otimes \cdots \otimes {\mathscr{N}}_{k}^{l_{k}-m_{k}}) \otimes
{\sF}_{\left(\frac{l_{1}}{r_{1}}, \ldots , \frac{l_{k}}{r_{k}}\right)} \;\;\mbox{projection formula}
\\
&\simeq &
\int {\sO}(n_{1}p_{1})^{\lfloor \frac{l_{1}-m_{1}}{r_{1}}\rfloor}
\otimes
\cdots
{\sO}(n_{k}p_{k})^{\lfloor\frac{l_{k}-m_{k}}{r_{k}}\rfloor}\otimes {\sF}_{\left(\frac{l_{1}}{r_{1}}, \ldots ,\frac{l_{k}}{r_{k}}\right)}
\\
&\simeq &
\int {\sF}_{\left(\frac{l_{1}}{r_{1}}-\lfloor \frac{l_{1}-m_{1}}{r_{1}}\rfloor, \ldots, \frac{l_{k}}{r_{k}}-\lfloor \frac{l_{k}-m_{k}}{r_{k}}\rfloor\right)}
\\
&\simeq &
{\sF}_{\left(\frac{m_{1}}{r_{1}},  \ldots , \frac{m_{k}}{r_{k}}\right)}.
\end{eqnarray*}

\section{The Parabolic Pullback}
\label{s:pp}

Consider a morphism $f:Y\rightarrow X$ of smooth projective curves. We obtain a diagram

\begin{center}
\beginpgfgraphicnamed{Dessin-g6}
\begin{tikzpicture}
\node (bl) at (0,0) {$Y$};
\node (br) at (2,0) {$X$};
\node (tl) at (0,2) {$Y_{f^*\lvD,\vec{r}}$};
\node (tr) at (2,2) {$X_{\lvD,\vec{r}}$};
\draw[->] (bl) --node[above =1pt] {$f$} (br);
\draw[->] (tl) --node[left=1pt] {$\pi_Y$} (bl);
\draw[->] (tl) --node[above=1pt] {$g$} (tr) ;
\draw[->] (tr) --node[right=1pt] {$\pi_X$} (br);
\end{tikzpicture}
\endpgfgraphicnamed
\end{center}

There are associated equivalences of categories
$$
\ef^X_{\lvD,\vec{r}} : \vect(X_{\lvD,\vec{r}}) \rightarrow \vect_\para(\lvD,\vec{r})
$$
and
$$
\ef^Y_{\lvD,\vec{r}} : \vect(Y_{\lvD,\vec{r}}) \rightarrow \vect_\para(\lvD,\vec{r}).
$$
Further there is an obvious pullback functor
$$
f^*:\vect_\para(\lvD,\vec{r}) \rightarrow \vect_\para(f^*\lvD,\vec{r}).
$$

\begin{prop}
\label{p:easypullback}
We have $f^*\circ \ef^X_{\lvD,\vec{r}} = \ef^Y_{f^*\lvD,\vec{r}} \circ g^*$.
\end{prop}

\begin{proof}
This is by flat base change.
\end{proof}

We will frequently apply the correspondence described in \ref{r:ms}, in what follows.

Set $\vec{r} = (r_{1}, \ldots, r_{k})$, $\lvD=(n_{1}p_{1}, \ldots, n_{k}p_{k})$ and $\vec{n}=(n_{1}, \ldots, n_{k})$. Consider an $\vec{r}$-divisible parabolic bundle $\sF_{\bullet}$ with parabolic structure along $\lvD.$ Using \ref{r:ms} we have a filtration
\[
F_{i,1} \supset \ldots \supset F_{i,m_{i}} \supset F_{i,m_{i+1}}=0
\]
and weights 
\[
0 \leq \alpha_{i,1}=\frac{s_{i1}}{r_{i}} <
\ldots < \alpha_{i,m_{i}}=\frac{s_{im_{i}}}{r_{i}} <1.
\]
Write $n_{i}s_{ij}=a_{ij}r_{i}+e_{ij}$ with $0 \leq e_{ij} <r_{i}.$ 
We also denote by $\sF_{ij}$ the preimage of $F_{ij}$ in 
$\sF_{(0,0,\ldots, 0)}$.
For $x \in \frac{1}{r_{i}}\lvZ\cap [0,1)$ define a subsheaf
$W_{ij}^{x}(\sF_{\bullet})$ of $\sF_{(0, \ldots, 0)}(n_{i}p_{i})$ by
\[
W_{ij}^{x}(\sF_{\bullet})=
\left\{
\begin{matrix}
\sF_{(0, \ldots, 0)}(a_{ij}p_{i})+\sF_{i,j+1}(n_{i}p_{i})
&
\mbox{if}\;\;x \leq \frac{e_{ij}}{r_{i}}
\\
\sF_{(0, \ldots, 0)}((a_{ij}-1)p_{i})+\sF_{i,j+1}(n_{i}p_{i}) &
\mbox{otherwise}
\end{matrix}
\right.
\]
We have a subsheaf 
$$
\sF^x_i = \bigcap_j W_{ij}^x(\sF_{\bullet})
$$
of $\sF_{(0, \ldots, 0)}(n_{i}p_{i}).$

When $x\ge 0$, we construct subsheaves 
$\sqrt[\vec{n}]{\sF_{\bullet}}_{(0, \ldots, 0, x, 0, \ldots, 0)}$
of $$\sF_{(0,\ldots,0)}(n_1p_{1}+\ldots + n_kp_k)$$ by setting
\[
\sqrt[\vec{n}]{\sF_{\bullet}}_{(0, \ldots, 0, x, 0, \ldots, 0)}
=(\cap_{j} W_{ij}^{x}(\sF_{\bullet})) + \sum_{i\ne k} \sF^0_k
=
\sF^x_i +\sum_{i\ne k}\sF^0_k,
\]
where the non-zero entry of the tuple is at the $i$th position. If $a_{i(j+1)} = a_{ij}$ then $e_{i,j+1} > e_{ij}$. Hence we have that $x \leq y$ implies 
\[
\sqrt[\vec{n}]{\sF_{\bullet}}_{(0, \ldots, 0, x,0,\ldots, 0)}
\supseteq 
\sqrt[\vec{n}]{\sF_{\bullet}}_{(0, \ldots, 0, y, 0, \ldots, 0)}.
\]

This extends to a uniquely to a parabolic bundle 
\[
\sqrt[\vec{n}]{\sF_{\bullet}}_{\bullet}:\left(\frac{1}{\vec{r}}\lvZ\right)^{\op} \rightarrow 
\vect(X).
\]
Setting $\frac{\vec{r}}{\vec{d}} = \left(\frac{r_{1}}{d_{1}}, \ldots, \frac{r_{k}}{d_{k}}\right)
$
where $d_{i} = \gcd (r_{i}, n_{i})$, note that this parabolic bundle is really $\frac{\vec{r}}{\vec{d}}$-divisible!

Set $\lvD_{red} = (p_{1}, \ldots, p_{k}).$ We have a diagram
\begin{center}
\beginpgfgraphicnamed{Dessin-g7}
\begin{tikzpicture}
 \node (tl) at (0,2) {$X_{(\mathbb{D}_{red},\frac{\vec{r}}{\vec{d}})}$} ;
 \node (tr) at (4,2) {$X_{(\mathbb{D},\vec{r})}$} ;
 \node (b) at (2,0) {$X.$};
 \draw [>=latex,->] (tl) --node[below=1pt] {$\pi$} (b);
 \draw [>=latex,->] (tl) --node[above=1pt] {$\alpha$} (tr);
 \draw [>=latex,->] (tr) --node[right=1pt] {$\pi_{n}$} (b);
\end{tikzpicture}
\endpgfgraphicnamed
\end{center}

There are associated equivalences 
\begin{center}
\beginpgfgraphicnamed{Dessin-g8}
\begin{tikzpicture}
 \node (l) at (0,0) {$\ef: \vect(X_{\lvD_{\rm red}, \vec{r}/\vec{d}})$};
 \node (r) at (4.2,0) {$\vect_{\para}(\lvD_{\rm red}, \vec{r}/\vec{d}):\eg$};
 \draw [->] (1.5,.1) -- (2.5,.1);
 \draw [->] (2.5,-.1) -- (1.5,-.1);
\end{tikzpicture}
\endpgfgraphicnamed
\end{center}
and
\begin{center}
\beginpgfgraphicnamed{Dessin-g9}
\begin{tikzpicture}
 \node (l) at (0,0) {$\ef_n: \vect(X_{\lvD, \vec{r}})$};
 \node (r) at (4.2,0) {$\vect_{\para}(\lvD, \vec{r}):\eg_n.$};
 \draw [->] (1.5,.1) -- (2.5,.1);
 \draw [->] (2.5,-.1) -- (1.5,-.1);
\end{tikzpicture}
\endpgfgraphicnamed
\end{center}

In the remainder of this section will be devoted to proving that for a vector bundle $\sF$ on 
$X_(\lvD, \vec{r})$ we have 
$$
\sqrt[\vec{n}]{\ef_n(\sF)}\cong \ef(\alpha^*(\sF)).
$$ 

In order to motivate the proof and understand the definition above we compute some examples.

\begin{example}

We assume that there is only one parabolic point $p$ with parabolic divisor $np$ having $r$-divisable 
weights. Also set $d= \gcd(r,n)$. Consider the root line bundle $\sN^w$ with $0<w<r$
on $X_{np,r}$. A calculation
shows that 
\begin{eqnarray*}
\ef_n(\sN^w) & : &  \frac{l}{r} \mapsto \sO(np)^{\lfloor \frac{w-l}{r} \rfloor} \\
\ef(\alpha^*\sN^w) & : &  \frac{dl}{r} \mapsto \sO(p)^{\lfloor \frac{nw-dl}{r} \rfloor}. 
\end{eqnarray*}
Let's compute $\sqrt[n]{\ef_n(\sN^w)}$.  Write $wn = ar + e$. The filtration of 
$\ef_n(\sN^w)_0$ is given by
$$
\sF_1 = \sO \qquad \sF_2 = \sO(-np)
$$ 
and the weight of $\sF_1$ is $w/r$. So
$$
W_1^x = \left\{ 
\begin{array}{cc}
 \sO(ap)      & 0\le x \le e/r \\
 \sO((a-1)p)  & e/r < x < 1.
\end{array}
\right.
$$
Hence 
$$
(\sqrt[n]{\ef_n(\sN^w)})_x = \left\{ 
\begin{array}{cc}
 \sO(ap)      & 0\le x \le e/r \\
 \sO((a-1)p)  & e/r < x < 1.
\end{array}
\right.
$$
which agrees with $\ef(\alpha^*\sN^w)$.

Let us compute a rank two example. Consider the bundle 
$$
\sN^{w_1}\oplus \sN^{w_2}
$$ with $0<w_1<w_2<r$. A calculation
shows that 
\begin{eqnarray*}
\ef_n(\sN^{w_1}\oplus \sN^{w_2}) & : &  
	\frac{l}{r} \mapsto \sO(np)^{\lfloor \frac{w_1-l}{r} \rfloor} \oplus 
						\sO(np)^{\lfloor \frac{w_2-l}{r} \rfloor} \\
\ef(\alpha^*(\sN^{w_1}\oplus \sN^{w_2})) & : &  
	\frac{dl}{r} \mapsto \sO(p)^{\lfloor \frac{nw_1-dl}{r} \rfloor} \oplus 
						\sO(np)^{\lfloor \frac{nw_2-dl}{r} \rfloor} \\
\end{eqnarray*}
Let's compute $\sqrt[n]{\ef_n(\sN^{w_1}\oplus \sN^{w_2})}$.  Write $w_jn = a_jr + e_j$. The filtration of 
$\ef_n(\sN^w)_0$ is given by
\begin{eqnarray*}
\sF_1 & = & \sO\oplus \sO \\
\sF_2 & = & \sO(-np)\oplus \sO \\
\sF_3 & = & \sO(-np)\oplus \sO(-np) \\
\end{eqnarray*}
and the weight of $\sF_j$ is $w_j/r$ when $j=1,2$. So
$$
W_1^x = \left\{ 
\begin{array}{cc}
 \sO(a_1 p)    \oplus \sO(np)      & 0\le x \le e_1/r \\
 \sO((a_1-1)p) \oplus \sO(np)      & e_1/r < x < 1.
\end{array}
\right.
$$
and
$$
W_2^x = \left\{ 
\begin{array}{cc}
 \sO(a_2 p)    \oplus \sO(a_2 p)      & 0\le x \le e_2/r \\
 \sO((a_2-1)p) \oplus \sO((a_2-1)p)      & e_2/r < x < 1.
\end{array}
\right.
$$
Notice that $a_1\le a_2$ and equality implies $e_1<e_2$.
So we see that $\sqrt[n]{\ef\alpha^*(\sN^{w_1}\oplus \sN^{w_2})}$ 
agrees with $\ef(\alpha^*\sN^w)$.

\end{example}

\begin{prop}
Let $\sF$ be a vector bundle on $X_{\lvD, \vec{r}}$. Then there is a canonical inclusion
\[
\pi_{*}\alpha^{*}\sF \subset \pi_{n*}\sF(n_{1}p_{1}+\ldots +n_{k}p_{k}) 
\]
\end{prop}
\begin{proof}
We denote the canonical line bundles on $X_{\lvD,\vec{r}}$ by 
\[
 \mathscr{N}_{1,\vec{n}}, \mathscr{N}_{2,\vec{n}}, \ldots, \mathscr{N}_{k,\vec{n}}.
\]

We have a diagram
\begin{center}
\beginpgfgraphicnamed{Dessin-g10}
\begin{tikzpicture}
 \node (bl) at (0,0) {$\sF$};
 \node (tl) at (0,2) {$\alpha_{*}\alpha^{*}\sF$} ;
 \node (tr) at (4,2) {$\alpha_{*}\alpha^{*}\left(\sF\otimes \mathscr{N}_{n_{1}}^{r_{1}} \otimes \ldots \otimes \mathscr{N}_{n_{k}}^{r_{k}}\right)$} ;
 \node (br) at (4,0) {$\sF\otimes \mathscr{N}_{1,\vec{n}}^{r_{1}} \otimes \ldots \otimes \mathscr{N}_{k,\vec{n}}^{r_{k}}$};
 \draw [>=latex,right hook->] (bl) -- (tl);
 \draw [>=latex,->] (tl) -- (tr);
 \draw [>=latex,->] (bl) -- (br);
 \draw [>=latex,right hook->] (br) -- (tr);
\end{tikzpicture}
\endpgfgraphicnamed
\end{center}
We apply $\pi_{\vec{n},*}$ to obtain a diagram
\begin{center}
\beginpgfgraphicnamed{Dessin-g11}
\begin{tikzpicture}
 \node (bl) at (0,0) {$\pi_{\vec{n},*}\sF$};
 \node (tl) at (0,2) {$\pi_{*}\alpha^{*}\sF$} ;
 \node (tr) at (4,2) {$\pi_{*}\alpha^{*}\left(\sF\otimes \mathscr{N}_{n_{1}}^{r_{1}} \otimes \ldots \otimes \mathscr{N}_{n_{k}}^{r_{k}}\right)$} ;
 \node (br) at (4,0) {$\pi_{\vec{n},*} \sF (n_1p_1 + n_2p_2 + \ldots + n_kp_k).$};
 \draw [>=latex,right hook->] (bl) -- (tl);
 \draw [>=latex,->] (tl) --node[above=1pt] {$\lambda$} (tr);
 \draw [>=latex,->] (bl) -- (br);
 \draw [>=latex,right hook->] (br) --node[right=1pt] {$\mu$} (tr);
\end{tikzpicture}
\endpgfgraphicnamed
\end{center}

The question is now local and is easily checked.

\end{proof}

\begin{thm}
\label{t:pullback}
We have
\[
\sqrt[\vec{n}]{(\ef_n\sF)_{\bullet}}_{\bullet} \simeq (\ef \alpha^{*}\sF)_{\bullet}.
\]
\end{thm}
\begin{proof}
We use \ref{r:ms}. Both are then subbundles of 
$\ef_n\sF_{\bullet}(n_{1}p_{1}+\ldots + n_{k}p_{k})$ and hence the question is once again local. 
We may assume that there is only one parabolic point. Applying \ref{p:local} and \ref{t:correspondence}
we can assume $(\ef_n\sF)_\bullet$ is of the form :
$$
\frac{l}{r}  \mapsto  (\sO(p)^{n(\lfloor \frac{w_1-l}{r} \rfloor)})^{\oplus\rho_1} \oplus \ldots \oplus 
 (\sO(p)^{n(\lfloor \frac{w_k-l}{r} \rfloor)})^{\oplus\rho_k}
$$
with $0\le w_1< w_2 < \ldots <  w_k<r.$ 
Pulling back root line bundles along the morphism 
$$
\alpha : X_{p,r/d} \rightarrow X_{np,r}
$$
we get $\alpha^*(\sN_n)= \sN_1^{(n/d)}$ where $d=\gcd(r,n)$.
Using \ref{p:push} it follows that $(\ef \alpha^{*}\sF)_{\bullet}$ is the parabolic bundle
$$
\frac{l}{r}  \mapsto  (\sO(p)^{(\lfloor \frac{nw_1-l}{r} \rfloor)})^{\oplus\rho_1} \oplus \ldots \oplus 
 (\sO(p)^{(\lfloor \frac{nw_k-l}{r} \rfloor)})^{\oplus\rho_k}.
$$
We need to compute $\sqrt[\vec{n}]{(\ef_{n}\sF)_{\bullet}}_\bullet$. We compute the value at $l=0$.
One can deduce the general result by shifting weights.
So we compute :
\begin{eqnarray*}
W_{1}^0((\ef_n \sF)_\bullet) &=& (\sO(p)^{(\lfloor \frac{nw_1}{r} \rfloor)})^{\oplus \rho_1}  
\oplus \sO(np)^{\oplus \rho_3} \oplus\ldots \oplus \sO(np)^{\oplus \rho_k}  \\
W_{2}^0((\ef_n \sF)_\bullet) &=& (\sO(p)^{(\lfloor \frac{nw_2}{r} \rfloor)})^{\oplus \rho_1}  \oplus 
(\sO(p)^{(\lfloor \frac{nw_2}{r} \rfloor)})^{\oplus \rho_2}  \oplus \sO(np)^{\oplus \rho_4} \ldots \oplus \sO(np)^{\oplus \rho_k}  \\
\vdots & & \vdots.
\end{eqnarray*}
Taking intersection we get
$$
\bigcap W_j^0 = (\sO(p)^{(\lfloor \frac{nw_1}{r} \rfloor)})^{\oplus\rho_1} \oplus \ldots \oplus 
 (\sO(p)^{(\lfloor \frac{nw_k}{r} \rfloor)})^{\oplus\rho_k}.
$$
which is what was needed.
\end{proof}

\section{The Cyclic Case}
\label{s:cyclic}

Given a one dimensional representation $V$ of $\lvZ/c\lvZ$ we call the integer $j$, $0\le j \le c-1$ the \emph{weight} of the representation if the generator $1+ c\lvZ$ acts by multiplication by  $e^{2\pi j\sqrt{-1}/c}$.

Suppose that $q:X \to Y$ is a $G$-cover, ramified at points $p_{1}, \ldots, p_{k}$ of $Y$. Suppose that the ramification index at
$p_i$ is $r_i$ and set  $\vec{r} = (r_{1}, \ldots, r_{k})$. 
Also, set $\mathbb{D} = (p_{1}, \ldots, p_{k}).$ By combining the results \ref{c:nori}, \ref{p:quotroot} and \ref{t:correspondence} we may view 
the cover as a tensor functor
$$
\sF_{q}: \Rep\text{-}G \rightarrow\vect_{\text{par}}(Y,\lvD, \vec{r}).
$$ 

If we choose preimages $q_i\in X$ of the $p_i$ we obtain cyclic subgroups $\lvZ/r_i\lvZ$ of $G$ that correspond to the stabilizers of 
$q_i$. We canonically identify the stabilizer with $\lvZ/r_i\lvZ$ by insisting that the stabilizer acts on the fiber of the sheaf $\sO(-q_i)$ at $q_i$ with weight one.

 Fix an irreducible representation $V$ of $G$. At each point $p_i$ we have a weight space decomposition of 
$$
V=\oplus_j W_j^i
$$
coming from the induced  action of the stabilizers $\lvZ/r_i\lvZ$. The spaces $W_j^i$ are representations of $\lvZ/r_i\lvZ$ and the generator of the group 
$\lvZ/r_i\lvZ$ acts by multiplication by $e^{2\pi j\sqrt{-1}/r_i}$. The numbers $j$ do not depend upon the choice of preimage $q_i$.

\begin{prop}
\label{p:genjump}
In the terminology of \ref{r:ms}, the weights of the $\sF_q(V)_\bullet$ at $p_i$ are $j/r_i$.
In other words, consider tuples
$$
I = (0,\ldots, 0, \underset{i\text{th}}{\frac{j}{r_{i}}}, 0, \ldots, 0) \qquad 
I' = (0,\ldots, 0, \underset{i\text{th}}{\frac{j+1}{r_{i}}}, 0, \ldots, 0).
$$
Then
\[
\sF_{q}(V)_I
=
\sF_{q}(V)_{I'}
\]
iff  $W^i_{j}=0$. 
\end{prop}
\begin{proof}
By Proposition \ref{p:quotroot} we have a diagram
\begin{center}
\beginpgfgraphicnamed{Dessin-g12}
\begin{tikzpicture}
 \node (b) at (2,0) {$Y$};
 \node (tl) at (-1,2) {$X$} ;
 \node (t) at  (2,2) {$\left[ X / G \right]$} ;
 \node (tr) at (5,2) {$Y_{(\mathbb{D},\vec{r})}$} ;
 \draw [>=latex,->] (tl) --node[right=3pt] {$\pi'$} (b);
 \draw [>=latex,->] (tl) -- (t);
\draw [>=latex,->] (t) --node[above=1pt] {$\sim$} (tr);
\draw [>=latex,->] (t) --node[right=1pt] {$\pi$} (b);
\end{tikzpicture}
\endpgfgraphicnamed
\end{center}
If $\sE$ is a $G$-equivariant bundle on $X$ which is the pullback of some $\tilde{\sE}$ on $\left[X/G\right]$, then 
$\pi_{*}(\tilde{\sE}) = \pi'_{*}(\sE)^{G}$.  
Set $D_i=\pi^*(p_i)_\red$. Hence
\[
\pi_{*}(\sN_{1}^{l_{1}} \otimes \ldots \otimes \sN_{k}^{l_{k}} \otimes \tilde{\sE}) = \pi'_{*}(\sO(l_{1}D_{1}) \otimes \ldots \sO(l_{k}D_{k}) \otimes \sE)^{G}.
\]
The question is now local. In formal neighbourhoods of $q_i$ and $p_i$ the morphism comes from a morphism of algebras of the form
\begin{eqnarray*}
 k[[t]] & \rightarrow & k[[s]] \\
 t &\mapsto & s^{r_i}.
\end{eqnarray*}
The group action is by multiplication by roots of unity. Computing invariants gives the result.
\end{proof}

Denote by $F_{m}$ a free group on the symbols $x_1,\ldots, x_m$. Consider the surjection
 $q:F_{m} \twoheadrightarrow \lvZ/c\lvZ$ that sends $x_i\mapsto 1$. 
There is an associated cover $X_q \rightarrow\mathbb{P}^{1}$ ramified possibly at 
$\{p_{1}, \ldots, p_{m}\}\cup \{\infty\}$ for some $p_{i} \in {\mathbb P}^{1}\setminus\{\infty\}$. Set 
$\vec{c}= (c, \ldots, c, \frac{c}{\gcd \{c,m\}})\in \mathbb{Z}^{m+1}$,
$\mathbb{D} = (p_{1}, \ldots, p_{m}, \infty),$ and 
$D= p_{1} + \ldots +p_{m}+\infty.$  For the remainder of this section $V_j$ will denote the one dimensional
representation of $\lvZ/c\lvZ$ where $1+c\lvZ$ acts by multiplication by $e^{2\pi j\sqrt{-1}/c}$.
Set 
\[
{\sF}_{X_{q}}(V_{j})_{(0, \ldots, 0)}=:
\sO(s_{j}),
\]
where $s_{j}$ is some integer. 
Also, let $w_{j}$ denote the rational number in $[0,1)$ which differs from $-\frac{mj}{c}$ by an integer. 

The purpose of this section is to describe the functor $\sF_{X_{q}}$. To this end, in the above proposition take $X=X_{q},$ 
$Y= \mathbb{P}^{1}$, $G= \mathbb{Z}/c\mathbb{Z},$ $k=m+1,$ $D_{j} = p_{j}$ for $1 \leq j \leq m$, $D_{m+1}=\infty$, and $\sF_q(V_j) = \sF_{X_{q}}(V_{j})_{\bullet}.$ This gives
\begin{cor}
\label{l:1}
Let $t = \frac{a}{\gcd(m,c)}$ and suppose $0 \leq t \leq w_{j}.$ Then
\[
\sF_{X_q}(V_{j})_{(0, \ldots, 0, t)} = \sO(s_{j}),
\]
and
\[
\sF_{X_q}(V_{j})_{(0, \ldots, 0, w_{j}+ \frac{\gcd(m,c)}{c})} = \sO(s_{j})(-\infty).
\]
Moreover,
if the non-zero entry of the tuple is at the $i$th position for $1 \leq i \leq m,$
\[
\sF_{X_q}(V_{j})_{(0, \ldots, 0, \frac{j+1}{c}, 0, \ldots, 0)} = \sO(s_{j})(-p_{i}),
\]
but
\[
\sF_{X_q}(V_{j})_{(0, \ldots, 0, \frac{j}{c}, 0, \ldots, 0)} = \sO(s_{j}).
\]
\end{cor}

Let $\delta_{ij}$ denote the Kronecker delta function.
\begin{lemma}
\label{l:keycalc}
 If $1\le w_1 + w_j$ then
$$
(\sF_{X_q}(V_1)_{\bullet}\otimes\sF_{X_q}(V_j)_{\bullet})_{(0,\ldots,0)}  
= \sO(s_1 + s_j + 1+m\delta_{c-1,j}).
$$
Otherwise, 
$$
(\sF_{X_q}(V_1)_{\bullet}\otimes\sF_{X_q}(V_j)_{\bullet})_{(0,\ldots,0)} =
\sO(s_1 + s_j+m \delta_{c-1,j}).
$$
\end{lemma}

\begin{proof}
Consider $t\in \frac{\gcd(m,c)}{c}\lvZ$ and set 
$$
\vec{t} = (0,\ldots, 0, t).
$$
Write $t=n+f$ where $f\in [0,1)$. We compute
$$
(\sF_{X_q}(V_1)_{\vec{t}}\otimes\sF_{X_q}(V_j)_{-\vec{t}}).
$$
The possibilities are 
$$
(\sF_{X_q}(V_1)_{\vec{t}}\otimes\sF_{X_q}(V_j)_{-\vec{t}})
=\left\{
\begin{array}{c}
\sO(s_1+s_j+1) \\
\sO(s_1+s_j) \\
\sO(s_1+s_j-1) \\
\sO(s_1+s_j-2) 
\end{array}\right.
$$
We are interested in when the first possibility occurs as the second occurs at $t=0$ so when we take the sheaf
generated by all possible tensor products the value will be at least this sheaf. 

Suppose that $1\le w_1 + w_j$. Now take $t=1-w_j$. Then
$$
\sF_{X_q}(V_j)_{-\vec{t}} = \sO(s_j +1).
$$
and 
$$
\sF_{X_q}(V_1)_{\vec{t}} = \sO(s_1).
$$

Conversely, suppose that  
$$
(\sF_{X_q}(V_1)_{\vec{t}}\otimes\sF_{X_q}(V_j)_{-\vec{t}}) = \sO(s_1+s_j+1).
$$
We have either
$$
w_1 - 1 \le w_j - 1 < w_1 \le w_j
$$
or
\[
w_j-1 \le w_1 -1 <w_j \le w_1.
\]
We conclude that $-f \le w_j-1$ and $f\le w_1$ or we must have
$-f \le w_1 -1$ and $f\le w_j$. 

We conclude that there is a $t$ for which 
$$
(\sF_{X_q}(V_1)_{\vec{t}}\otimes\sF_{X_q}(V_j)_{-\vec{t}}) = \sO(s_1 + s_j +1)
$$
if and only if $w_1+w_j\ge 1$.

Now we turn our attention to the other parabolic points. We preserve the notation above except we set 
$$
\vec{t} = (0,\ldots, 0, t,0,\ldots, 0)
$$
and now $t\in \frac{1}{c}\lvZ$. 
We have a chain of inequalities
$$
\frac{1}{c}-1 \le \frac{j}{c}-1 < \frac{1}{c} \le \frac{j}{c}.
$$
Suppose firstly that $j <c-1.$ Then if $-f\le \frac{j}{c}-1$ we have $f\ge 1 - \frac{j}{c} > \frac{1}{c}$.
If $-f =  \frac{1}{c} - 1$ then $f> \frac{j}{c}$. It follows that 
$$
(\sF_{X_q}(V_1)_{\vec{t}}\otimes\sF_{X_q}(V_j)_{-\vec{t}}) = \sO(s_1+s_j).
$$
When $j<c-1,$ the result follows by putting this together.

Now fix $j=c-1$. Set
\[
\vec{u}=(u_1, \ldots, u_m, u_{m+1})
\]
where $u_i\in \frac{1}{c}\lvZ$ for $1\leq i\leq m$ and $u_{m+1} \in \frac{\gcd{(m,c)}}{c},$
 and write $u_i = n_i + f_i$ where $f_i \in [0,1).$

Here, computing 
\[
\sF_{X_{q}}(V_1)_{\vec{u}} \otimes\sF_{X_q}(V_{c-1})_{-\vec{u}}
\]
the possibilities are 
\[
\sO(s_1+s_{c-1}+g(\vec{u}))
\]
where $g(\vec{u})$ ranges over all integers from $-2$ to $m+1:$ Indeed, as before, the parabolic point at infinity gives at most a contribution of $+1$ to $g(\vec{u})$ and at least $-2,$ while each finite parabolic point contributes either 0 or $+1$.

At the same time,
\begin{equation}
\label{e:c-1}
\sF_{X_q}(V_1)_{(\frac{1}{c},\ldots,\frac{1}{c},0)}\otimes \sF_{X_q}(V_{c-1})_{(-\frac{1}{c},\ldots,-\frac{1}{c},0)}
=\sO(s_1+s_{c-1}+m).
\end{equation}
This means that
\[
(\sF_{X_q}(V_1)_{\bullet}\otimes\sF_{X_q}(V_{c-1})_{\bullet})_{(0\ldots,0)} \supseteq \sO(s_1 + s_{c-1} + m)
\]
from the definition of parabolic tensor product.

Hence, we need only determine when $g(\vec{u})=m+1.$

Suppose that $1 \leq w_1+w_{c-1}.$
Then if  $\vec{u} = (\frac{1}{c}, \ldots, \frac{1}{c}, 1-w_{c-1})$,
$$
\sF_{X_q}(V_{c-1})_{-\vec{u}} = \sO(s_{c-1}+m +1)
$$
and 
$$
\sF_{X_q}(V_1)_{\vec{u}} = \sO(s_1).
$$
Conversely, suppose that there exists some $\vec{u}$ such that
\[
\sF_{X_{q}}(V_1)_{\vec{u}} \otimes\sF_{X_q}(V_{c-1})_{-\vec{u}} 
=
\sO(s_1 +s_{c-1}+m+1).
\]
This case only occurs when either $-f_{m+1}\leq w_{c-1}-1$ and $f_{m+1} \leq w_1$ or $-f_{m+1}\leq w_1 -1$ and $f_{m+1}\leq w_{c-1}$ by the same argument as before. Hence, necessarily, 
$w_1+w_{c-1} \geq 1$.
\end{proof}

\begin{remark}
\label{r:powers}
$\sF_{X_q}(V_{j})_{\bullet}$ is the $j$th parabolic tensor power of   $\sF_{X_q}(V_{1})_{\bullet}$: Indeed, since
$\sF_{X_q}$ is a tensor functor, we must have $\sF_{X_q}(V_{1})_{\bullet}^{\otimes c} =\sF_{X_q}(V_{1}^{\otimes c})_{\bullet} = \sF_{X_q}(V_{0})_{\bullet},$ the trivial parabolic bundle. Similarly, $\sF_{X_q}(V_{1})_{\bullet}^{\otimes l} = \sF_{X_q}(V_{j})_{\bullet}$ whenever 
$l \equiv j$ modulo $c$.

In order to determine $\sF_{X_q}(V_j)_\bullet$ it thus suffices to compute $s_{1}$.
\end{remark}

For each $j$ with $1 \leq j \leq c-1,$ set
\[
\kappa_{m,c}^{(j)}=\left\{\begin{array}{cc}
1 & \text{when }w_1+w_{j} \geq 1
\\
0 & \text{otherwise}
\end{array}
\right.
\]
and
\[
\kappa_{m,c} 
=
\sum_{j=1}^{c-1}\kappa_{m,c}^{(j)}
= \left|\{j: 1 \leq j \leq c-1, w_1 +w_j\geq 1\}\right|.
\]

\begin{thm}
\label{t:cyclic} With notation as above, 
\[
s_{1}=
-\frac{m+\kappa_{m,c}}{c}
\]
\end{thm}
\begin{proof}

Applying Lemma \ref{l:keycalc} iteratively, along with Remark \ref{r:powers}, one finds that
\[
\sO(s_{c-1})
=
\sO((c-1)s_1+\kappa_{m,c} - \kappa_{m,c}^{(c-1)}).
\]

Repeat the calculation once more (in the special case that $j=c-1$) to obtain
\[
\sO(s_c)
=
\sO(cs_1+\kappa_{m,c} +m).
\]
The result now follows.
\end{proof}

The proof of \ref{t:cyclic} yields the
\begin{cor} 
\label{c:sj}
For $1 \leq j \leq c-1,$ the $s_{j}$ of  \ref{l:1} are given in terms of $s_{1}$ by
\[
s_j = js_1 +\sum_{i=1}^{j-1}\kappa_{m,c}^{(i)}  = -j\left(\frac{m+\kappa_{m,c}}{c}\right) + \sum_{i=1}^{j-1}\kappa_{m,c}^{(i)}.
\]
\end{cor}

\begin{cor}
\label{c:bounding}
We have $s_0= 0$ and $s_j\le -1$ for $j>0$. 
\end{cor}

\begin{proof}
The assertion for $s_0$ is clear. 
The numbers are necessarily integers. We have, by definition $s_1<0$ and hence 
$s_1\le -1$. The result now follows.
\end{proof}

By the above computation, $\kappa_{m,c}$ is necessarily congruent to $-m$ modulo $c.$ This fact may be shown independently:
\begin{lemma}
\[
\kappa_{m,c} \equiv -m \text{ modulo }c.
\]
\end{lemma}
\begin{proof}
When $m\equiv 0$ modulo $c$, it follows that $w_j = 0$ for all $1 \le j \le c-1,$ and hence $\kappa_{m,c} =0.$

Suppose now that $m\equiv -v$ modulo $c$, for some $0 <v<c.$
Then $w_1 = \frac{v}{c},$ while for $j$ with $1 \le j \le c-1,$
\[
w_j = \left\{
\begin{array}{cc}
\frac{vj}{c} &  0 < vj < c
\\
\vdots & \vdots \\
\frac{vj-tc}{c} & tc \le vj < (t+1)c
\\
\vdots & \vdots \\
\frac{vj - (v-1)c}{c} & (v-1)c \le vj < vc.
\\
\end{array}
\right.
\]
For $t$ with $0 \le t \le c-1,$ then $tc \le vj < (t+1)c$ implies $0 \le vj-tc <c.$ Now let $j_t$ be the largest integer value of $j$ satisfying this inequality. Then $v(j_t+1) -tc\geq c,$ so that 
\[
w_1+w_{j_t} = \frac{v(1+j_t) - tc}{c}
\geq 1.
\]
At the same time, for any integer $j$ satisfying the inequality which also has $j <j_t,$ then $j+1 \le j_t$ and 
necessarily
\[
w_1+ w_{j} \le \frac{vj_t -tc}{c} <1.
\]
So among the integers $j$ such that $tc \le vj <(t+1)c$, there is exactly one with $w_1 +w_j \geq 1.$ There are exactly $v$ such inequalities, so $\kappa_{m,c} = v.$
\end{proof}

\section{Reduction to the cyclic case}
Suppose that $X_{q} \rightarrow \mathbb{P}^{1}$ is a {Galois} covering with $\Deck(X_{q}/\mathbb{P}^{1}) = G$ ramified at 0,1 and $\infty.$ Let 
$q: F_{2} \twoheadrightarrow G$ denote the corresponding surjection and $\mathbb{T} = (0,1,\infty).$ Then as before, by \ref{c:nori}, \ref{p:quotroot} and \ref{t:correspondence} the cover may be viewed as a functor
\[
F_{X_q}:\Rep\text{-}G \rightarrow \vect_{\para}(\mathbb{P}^{1}, \mathbb{T}).
\]
Our goal in this section is to produce a bound on the $u_{j}$ for which 
\[
F_{X_q}(V)_{(0, \ldots, 0)} = \sO(u_{1})\oplus \ldots \oplus \sO(u_{k})
\]
for a fixed $V \in \Ob(\Rep\text{-}G).$

The idea is to reduce to the cyclic case by delooping the ramification at 0 as follows:
 Suppose that the ramification index at 0 is $m$ - i.e. under the mapping $q$, the image of the generator of $F_{2}$ 
corresponding to a loop about 0 in $\pi_{1}(\mathbb{P}^{1})$ has order $m$ in $G.$  Form the base change 
\begin{center}
\beginpgfgraphicnamed{Dessin-g13}
\begin{tikzpicture}
 \node (tl) at (0,2) {$X_{q} \times_{\mathbb{P}^{1}} \mathbb{P}^{1}$};
 \node (tr) at (4,2) {$X_{q}$} ;
 \node (bl) at (0,0) {$\mathbb{P}^{1}$} ;
\node (br) at (4,0) {$\mathbb{P}^{1}$};
 \draw [>=latex,->] (tl) -- (tr);
 \draw [>=latex,->] (tl) -- (bl);
\draw [>=latex,->] (bl) --node[above=1pt] {$z\mapsto z^{m}$} (br);
\draw [>=latex,->] (tr) -- (br);
\end{tikzpicture}
\endpgfgraphicnamed
\end{center}
and denote the desingularization of $X_{q} \times_{\mathbb{P}^{1}} \mathbb{P}^{1}$ by $Y$.
 Now $Y \rightarrow \mathbb{P}^{1}$ ramifies at $\infty$ and the $m$th roots of unity, $\mu_{m}.$ 
Hence $Y$ corresponds to a homomorphism $h: F_{m} \rightarrow G$ which factors through $F_{2}$ by mapping the generators of 
$F_{m}$ corresponding to each root of unity to the generator $\sigma_{1}$ of $F_{2}$ corresponding to 1.

Then the image of $h$ is generated by $q(\sigma_{1}),$ which is a cyclic subgroup of $G$, say $\mathbb{Z}/c\mathbb{Z}.$

We have a decomposition
$Y = \coprod_{\tau \in G/{\rm Im}(h)} Y_{\tau}$ where the $Y_{\tau}$ are all cyclic covers.

Using the argument at the start of \S\ref{s:cyclic}, we obtain a tensor functor
$$
F_Y :\Rep\text{-}G \rightarrow \vect_{\para}(\mathbb{P}^{1}, (\mu_{m}, \infty)).
$$

\begin{lemma}
 The functor $F_Y$ factors as
\begin{center}
\beginpgfgraphicnamed{Dessin-g14}
\begin{tikzpicture}
 \node (tl) at (0,2) {$\Rep\text{-}G$};
 \node (tr) at (4,2) {$\vect_{\para}(\mathbb{P}^{1}, (\mu_{m}, \infty))$} ;
 \node (bl) at (0,0) {$\Rep\text{-}\lvZ/c\lvZ$};
 \draw [>=latex,->] (tl) --node[above=1pt] {$F_Y$} (tr);
 \draw [>=latex,->] (tl) -- (bl);
 \draw [>=latex,->] (bl) --node[below=1pt] {$F_{Y_e}$} (tr);
\end{tikzpicture}
\endpgfgraphicnamed
\end{center}
\end{lemma}

\begin{proof}
 The functors are computed by taking invariants as in the proof of \ref{p:genjump}. The result
now follows from the disjoint union above.
\end{proof}

We need :
\begin{prop}
\label{p:a_0}
If $\lvD = (p_1, \ldots, p_k)$ with $\vec{r} = (r_1, \ldots, r_k)$, and $\lvD'=(p_0,p_1, \ldots, p_k)$ with $\vec{r}' = (1, r_1, \ldots, r_k),$ then there exist natural equivalences of tensor categories
\begin{center}
\beginpgfgraphicnamed{Dessin-g15}
\begin{tikzpicture}
\node (l) at (0,0) {$\ef ': \vect_{\para}(\lvD', \vec{r}')$};
\node (r) at (4,0) {$\vect_{\para}(\lvD, \vec{r}): \eg'.$};
\draw [>=latex,-> ] (1.5,.1) -- (2.5,.1);
\draw [>=latex,->] (2.5,-.1) -- (1.5,-.1);
\end{tikzpicture}
\endpgfgraphicnamed
\end{center}
\end{prop}
\begin{proof}
The root stacks $X_{\lvD, \vec{r}}$ and $X_{\lvD', \vec{r}'}$ are isomorphic. Now invoke Theorem \ref{t:correspondence}.
\end{proof}

\begin{remark}
\label{r:a_0}
Let $\zeta_m$ denote a primitive $m$th root of unity. Then in the notation of \ref{p:a_0} set $\lvD = (\zeta_m, \zeta_m^2, \ldots, \zeta_m^{m-1}, 1, \infty)$ and $\vec{r} = (c, \ldots, c, \frac{c}{\gcd(m,c)}).$ Also take $p_0 = 0.$ By \ref{p:res} and \ref{t:pullback} we have that $f^*_{\para}(F_{X_q}) = \eg' F_Y.$ 

Since $\eg'$ is an equivalence of tensor categories, the constants computed in section\ref{s:cyclic} pertaining to $F_Y$ are the same as those relating to $\eg'F_Y.$
\end{remark}

We denote by $\kappa_{m,c}$ and $\kappa_{m,c}^{(i)}$ the numbers defined before Theorem \ref{t:cyclic} for the cover $Y_e \rightarrow \mathbb{P}^1$. We will also make use of the notation set up after \ref{p:easypullback}. In particular, let $a_1$ denote the minimum among the $a_{i1}$. Further denote by $a_0$ and $a_{\infty}$ 
 $a_{i1}$ for the index $i$ corresponding to the points 0 and $\infty$ respectively.

The representation $V$ viewed as a representation of $\lvZ/c\lvZ$ decomposes into weight spaces :
$$
V = V_{j_1} \oplus \ldots \oplus V_{j_k}.
$$
We have 
$$
F_{Y_e}(V)_{(0,\ldots, 0)} = \sO(t_1) \oplus \ldots \oplus \sO(t_k)
$$
where the $t_i$ are computed in \ref{t:cyclic} and \ref{c:sj}. We may reindex so that
$$
t_1 \le t_2 \le \ldots t_k \le 0.
$$
The last inequality is by \ref{c:bounding}.

\begin{thm}
\label{t:bound}
With the above notation, consider
\[
F_{X_q}(V)_{(0,\ldots,0)}=\sO(u_{1})\oplus \ldots \oplus \sO(u_{k}).
\]
We reindex so that
$$
u_1 \le u_2 \le \ldots \le u_k.
$$
Then the $u_j$ are bounded above as follows:
\[
u_j \leq  \frac{t_j}{m} - \frac{a_0}{m}  - \frac{a_{\infty}}{m}.
\]
(Hence the $u_j$ are negative, by \ref{c:bounding}.)
\end{thm}

\begin{proof} 
 We have 
\[
f^{*}(F_{X_q}(V)_{(0, \cdots, 0)}) = \sO(mu_{1}) \oplus \ldots \oplus \sO(mu_{k}).
\]
With $\zeta_m$ denoting a primitive $m$th root of unity as above,
the curve $Y$ ramifies over 
$$
p_1 =\zeta_m , \ldots, p_m = \zeta_m^m=1, p_{m+1}=\infty.
$$
By \ref{r:a_0} the parabolic pullback of $F_{X_q}(V)_{\bullet}$ also has 1-divisibility at $p_0:=0.$

Now by the definition of parabolic pullback,
$ f^{*}_{\para}F_{X_q}(V)_{(0, \ldots, 0)} $
contains the intersection 
$\cap_{j}W_{ij}^{0}$. Hence
$$
f^{*}_{\para}F_{X_q}(V)_{(0, \ldots, 0)} \supseteq(f^{*}(F_{X_q}(V)_{(0, \cdots, 0)})(a_{i1}))
$$
as $a_{i1}\le a_{ij}$. Notice that
$$
a_{11} = \ldots a_{m1} = a_1.
$$
Hence
\begin{eqnarray*}
&&
\sO(mu_{1})\oplus \ldots \oplus \sO(mu_{k})(a_0.0 + a_\infty.\infty + \sum a_{1}p_{i})
\\
&\simeq&
\sO(mu_{1}+a_0+ ma_{1}+a_{\infty})\oplus \ldots \oplus \sO(mu_{k}+a_0+ma_{1}+a_{\infty})\\
&\subseteq&
f^{*}_{\para}F_{X_q}(V)_{(0,\ldots, 0)} \\
& = & \sO(t_1) \oplus \ldots \oplus \sO(t_k).
\end{eqnarray*}
The result now follows from  \ref{l:inc} below and observing that $a_1=0$.
\end{proof}

\begin{lemma}
\label{l:inc}
If $\sO(s_{1}) \oplus \ldots \oplus \sO(s_{u}) \subseteq
\sO(t_{1}) \oplus \ldots \oplus \sO(t_{u}),$ there exists $\sigma \in S_{u}$ such that $s_{\sigma({j})} \leq t_{j}$ for all $j$ with
$1 \leq j \leq u.$
\end{lemma}
\begin{proof}
When $u=1,$ this is well-known. Proceeding by induction, suppose that the assertion is known to be valid for all $u \leq N-1.$ Then consider an injection
\[
\phi: \sO(s_{1}) \oplus \ldots \oplus \sO(s_{N}) \hookrightarrow
\sO(t_{1}) \oplus \ldots \oplus \sO(t_{N})
\]
where the $s_{j}$ and $t_{j}$ may be taken to be ordered - i.e. $s_{1} \leq \ldots \leq s_{N}$ and $t_{1} \leq \ldots \leq t_{N}$. Necessarily, $s_{N} \leq t_{L}$ for some $L$, but if $s_{N} \leq t_{1}$ we are done. Suppose then that there exists some $i$ such that $t_{i-1} < s_{N} \leq t_{i}.$ For $j$ with $i \leq j \leq N,$ consider the mapping
\[
\phi_{j} : \sO(s_{1}) \oplus \ldots  \oplus \sO(s_{N-1})
\rightarrow \sO(t_{1}) \ldots \oplus \hat{\sO(t_{j})}
\oplus \ldots\sO(t_{N})
\]
induced from $\phi.$
Should there exist $j$ for which $\phi_{j}$ is injective, we are done by the inductive hypothesis. Suppose to the contrary that for every $j$, $\phi_{j}$ is not injective. Then we can show this implies the original $\phi$ could not have been injective: Indeed, $s_{N}>t_{i-1}$ implies that under $\phi$, the restricted morphism  $\sO(s_{N}) \rightarrow \sO(t_{1}) \oplus \ldots \oplus \sO(t_{i-1})$ is zero. 

Passing to the generic point of the curve the morphism $\phi$ is given by an $N\times N$ matrix.
The last row of this matrix begins with $i-1$ zero entries.
Computing the determinant of $\phi$ by cofactor expansion along this row, we find
\[
\det \phi = 0 + \det\phi_{i} \cdot \gamma_{i}
+ \ldots + \det \phi_{N} \cdot \gamma_{N}
\]
for some constants $\gamma_{j}$. Hence the morphism at the generic point is not injective. This is a contradiction as pullback to the generic point is flat.
\end{proof}

\begin{example}
\label{e:improve}
Denote by $Q_8$ the quaternion group of order 8. It has a two dimensional representation 
given in terms of matrices by

\begin{center}
\begin{tabular}{ccc}\vspace{.1cm}   
$i$ & $\mapsto$ &$\left( \begin{array}{cc} \sqrt{-1} & 0 \\ 0 & \sqrt{-1}  \end{array} \right)$\\
\vspace{.1cm}
$j$ & $\mapsto$ &$\left( \begin{array}{cc} 0 & 1 \\ -1 & 0  \end{array} \right)$\\  
$k$ & $\mapsto$ &$\left( \begin{array}{cc} 0 &\sqrt{-1}  \\  \sqrt{-1} & 0  \end{array} \right)$\\  
\end{tabular}
\end{center}
Consider the quotient $F_2\twoheadrightarrow Q_8$ with $x_0\mapsto j$, $x_1 \mapsto i$. 
As $x_1$ has a weight 3 eigenspace we have $t_1=-3$. Both $a_1$ and $a_\infty$ are 1.
Hence $u_1\le -2$.
\end{example}

It follows from the lower bound in \cite[theorem 5.12]{borne:07} that $u_1$ must be -2. 

\bibliographystyle{plain}	
\bibliography{mybib,mybib2}

\end{document}